\documentclass[a4paper, 11pt, reqno, twoside]{amsart}
\numberwithin{equation}{section}

\usepackage[T1]{fontenc}
\usepackage[latin1]{inputenc}

\usepackage{graphicx}
\usepackage{epstopdf}

\usepackage{color}

\usepackage{amssymb,amscd,amsmath}

\newcommand{\realspan}{span_{{\mathbb R}}\,}
\newcommand{\compspan}{span_{{\mathbb C}}\,}

\newcommand{\LL}{\mathcal{L}}

\newcommand{\bigOh}{\mathcal O}

\newcommand{\R}{\mathbb R}
\newcommand{\s}{\mathbb S}

\newcommand{\Z}{\mathbb Z}
\newcommand{\Schwartz}{\mathcal{S}}
\newcommand{\Lis}{\mathcal{L}_\text{is}}
\newcommand{\BL}{\mathcal{L}}

\newcommand{\Sn}{\mathcal{S}_N}

\DeclareMathOperator{\Diff}{D}
\DeclareMathOperator{\Lop}{L}
\DeclareMathOperator{\Nop}{N}
\DeclareMathOperator{\Fop}{F}
\DeclareMathOperator{\Piop}{\Pi}
\DeclareMathOperator{\F}{{\mathcal F}}
\DeclareMathOperator{\id}{id}

\DeclareMathOperator*{\range}{ran}
\DeclareMathOperator{\spn}{span}

\newtheorem{theorem}{Theorem}[section]
\newtheorem{lemma}[theorem]{Lemma}
\newtheorem{corollary}[theorem]{Corollary}
\newtheorem{proposition}[theorem]{Proposition}

\newtheorem*{acknowledgments}{Acknowledgments}

\theoremstyle{definition}
\newtheorem{remark}[theorem]{Remark}

\title[Global bifurcation for the Whitham equation]{\small Global bifurcation for the Whitham equation}

\author{Mats Ehrnstr\"om}
\address{Department of Mathematical Sciences, Norwegian University of Science and Technology, 7491 Trondheim, Norway}
\address{Institut f\"ur Angewandte Mathematik, Leibniz Universit\"at Hannover, Welfengarten 1, 30167 Hannover, Germany}
\email{mats.ehrnstrom@math.ntnu.no}
\author{Henrik Kalisch}
\address{Department of Mathematics, University of Bergen, Postbox 7800, 5020 Bergen, Norway}
\email{henrik.kalisch@math.uib.no}

\begin{document}

\maketitle

\begin{abstract}
We prove the existence of a global bifurcation branch of $2\pi$-periodic, 
smooth, traveling-wave solutions of the Whitham equation. 
It is shown that any subset of solutions in the global branch contains 
a sequence which converges uniformly to some solution of 
H\"older class $C^{\alpha}$, $\alpha < \frac{1}{2}$. 
Bifurcation formulas are given, as well as some properties along the 
global bifurcation branch. In addition, a spectral scheme for 
computing approximations to those waves is put forward, 
and several numerical results along the global bifurcation branch are presented, including the presence of a turning point and a `highest', cusped wave. 
Both analytic and numerical results 
are compared to traveling-wave solutions of the KdV equation.
\end{abstract}

\section{Introduction}
\noindent
The Whitham equation,
\begin{equation}\label{eq:whit}
\eta_{t} + {\textstyle \frac{3}{2}\frac{c_0}{h_0}} \, \eta \, \eta_{x} 
         + K_{h_0} * \eta_{x} = 0,
\end{equation}
combines a generic nonlinear quadratic term with the exact 
linear dispersion relation for surface water waves on finite depth. 
Here, the kernel $K_{h_0} := \F^{-1} \left( c_{h_0} \right)$ is the inverse 
Fourier transform of the phase speed
\begin{equation}\label{eq:disp}
c_{h_0}(\xi) := \sqrt{{\textstyle \frac{g \tanh{(\xi h_0)}}{\xi} }}
\end{equation}
for the linearized water-wave problem; 
the constants $g$, $h_0$ and $c_0 := \sqrt{g h_0}$ denote, respectively, 
the gravitational constant of acceleration, the undisturbed water depth, 
and the limiting long-wave speed. 
The function $\eta(t,x)$ describes the deflection of the fluid surface 
from the rest position at a point $x$ at time $t$ \cite{MR1699025}. 
The Whitham equation was introduced by Whitham (see \cite{0163.21104}) 
as an alternative to the Korteweg--de Vries (KdV) equation, 
which describes the evolution of predominantly uni-directional, 
small-amplitude surface  waves in shallow water \cite{MR2196497,MR795808,MR1780702}. 
The KdV equation appears from \eqref{eq:whit} if one replaces 
$c_{h_0}(\xi)$ with its second-order approximation at $\xi = 0$, 
\begin{equation}\label{eq:dispKdV}
c_0 - {\textstyle \frac{1}{6}} c_0 h_0^2 \xi^2,
\end{equation}
and this explains the main motivation behind the Whitham equation. 
In particular, since $\xi=\frac{2 \pi}{\lambda}$ may be interpreted 
as a wave number, $\lambda$ being a typical wavelength, 
one sees from \eqref{eq:dispKdV} that the KdV equation 
is a poor approximation to \eqref{eq:disp} for large wave numbers. 

The Whitham equation \eqref{eq:whit} with the kernel \eqref{eq:disp} 
has some very interesting mathematical features. 
First of all, it is generically nonlocal, making pointwise estimates difficult. 
Moreover, $c_{h_0}(\xi)$ has slow decay, and the kernel $K_{h_0}$ is singular 
(it blows up at $x = 0$). This makes the Whitham equation in some important 
respects different from many other equations of the form \eqref{eq:whit}. 
One example is the possibility of singularities in solutions: 
the Whitham equation features (phenomenological) wave-breaking \cite{0802.35002} 
--- i.e. bounded solutions with unbounded derivatives --- 
and is conjectured to admit solutions with cusps \cite{0163.21104}. 

In this paper we consider steady solutions of the Whitham equation, i.e., traveling-wave solutions characterized by a constant speed and shape. The existence of smooth, small-amplitude, periodic traveling-wave solutions 
was established in \cite{EK08}, and their properties numerically investigated. 
The numerical calculations pursued in that investigation indicate 
a global branch of solutions approaching a `highest', cusped, wave --- 
as well as a sequence of periodic waves converging to a solitary wave. 
An analytic proof of the latter fact is given in \cite{EGW11}, for small waves, using a variational approach developed mainly in~\cite{Buffoni04a,GrovesWahlen10a}. The aim of this note is: 
(i) to place the Whitham equation in a general 
functional-analytic framework, in which the bifurcation theory for that 
and many related equations is quite natural; 
(ii) to extend the main local bifurcation branch found in \cite{EK08} to a global one; 
and (iii) to provide a spectral scheme for efficient calculation of Whitham waves. 
In addition, we prove some convergence results along the global branch of solutions, 
and several numerical calculations are given, complementing the more abstract exact theory. In particular, numerical evidence is given for a turning point along the global bifurcation branch, as well as additional calculations supporting the existence of a cusped wave at the end of the bifurcation branch.

Judging from the forms of the KdV and Whitham equations, 
one would expect several similarities for small waves of small wavelength 
(one such convergence result is given in \cite{EGW11}). 
To illustrate this, we perform a comparative local bifurcation 
analysis for the two equations, analytically and numerically. 
Although the solutions arise in the same manner from a 
curve of vanishing solutions, they develop differently along the corresponding 
bifurcation branches. We show that the Whitham solutions are all smooth and subcritical 
(their normalized wave speed is uniformly bounded away from unity), 
and that they converge uniformly to a wave of $C^\alpha$-regularity, 
$\alpha < \frac{1}{2}$. The intricate form of the Whitham kernel has yet 
hindered us from proving that the limiting wave is cusped and non-trivial. 
To see the difficulty, note that the Whitham kernel $K_{h_0}$ 
is not easily seen to be positive definite (cf., e.g., \cite{MR1823914}). 

This is the disposition of the paper: in Section~\ref{sec:prel} 
we normalize the equation for traveling wave-solutions and briefly discuss 
Fourier multipliers on H\"older spaces. Section~\ref{sec:local} contains 
a functional-analytic formulation of the problem, and local bifurcation theory 
for the Whitham and the KdV equations. Section~\ref{sec:global} 
is the main analytic section, in which we state and prove the global results. 
The sections~\ref{sec:spectral} and \ref{sec:results} are devoted to 
the numerical analysis. In Section~\ref{sec:spectral} we describe 
a spectral scheme for the Whitham equation, and use it to calculate wave profiles. 
Finally, Section~\ref{sec:results} takes a look at the shape of the bifurcation
branches, and provides a numerical comparison 
of the KdV and Whitham waves. 

\section{Some preliminaries}
\label{sec:prel}
Throughout the paper, the operator $\F$ will denote the extension to the space of 
tempered distributions $\Schwartz^\prime(\R)$ of the Fourier transform 
\[
\F(f)(\xi) := \int \hat f(x) \exp(-ix\xi)\,dx 
\]
on the Schwartz space $\Schwartz(\R)$, 
with inverse $\F^{-1}(f)(x) := \frac{1}{2\pi} \int \hat f(\xi) \exp(ix\xi)\,d\xi$. 

\subsection*{\sc Traveling waves.} 
By considering steady solutions with wave speed $c > 0$,
\[
\varphi(x-ct) := \eta(t,x), 
\]
the left-hand side of~\eqref{eq:whit} may be integrated to 
$-c \varphi + \frac{3 c_0}{4 h_0} \varphi^2 + K_{h_0} \ast \varphi$. 
In the following, we keep with the convention in denoting the 
steady variable $x -ct$ again by $x$. 
The scalings $\frac{3}{4 h_0} \varphi \mapsto \varphi$ and $x \mapsto h_0 x$ 
then yield the normalized equation
\begin{equation}\label{eq:canonical}
-\mu \varphi + \varphi^2 + K \ast \varphi = 0, 
\end{equation}
where $\mu := c/c_0$ is the non-dimensional wave speed, 
and we have set the constant of integration equal to zero
(see Remark~\ref{rem:galilean} below for a discussion of this). 
The symbol of the convolution operator in~\eqref{eq:canonical} is given by
\begin{equation}\label{eq:K1}
\F \left( K \right)(\xi) := {\textstyle \left(  \frac{\tanh\left( \xi \right)}{\xi} \right)^{1/2}}
\end{equation} 
for the Whitham equation. 
The corresponding Fourier multiplier for the KdV equation has the form
$1 - \frac{1}{6} \xi^2$, and the equation reads 
\begin{equation}\label{eq:kdv}
(1 - \mu) \varphi + \varphi^2 + \frac{1}{6} \varphi^{\prime\prime} = 0. 
\end{equation} 
Any solution of~\eqref{eq:canonical} with the kernel given by~\eqref{eq:K1} 
will be called a \emph{Whitham solution}, and any solution of~\eqref{eq:kdv} will be called a \emph{KdV solution}. 
We shall be dealing only with $2\pi$-periodic solutions.

\begin{remark}\label{rem:galilean}
The normalization $B = 0$ in
\[
-\mu \varphi + \varphi^2 + K \ast \varphi = B,
\]
cf.~\eqref{eq:canonical}, is a choice of convenience: both the Whitham and KdV equations are invariant under the Galilean transformation 
\[
\varphi \mapsto \varphi + \gamma, \qquad \mu \mapsto \mu + 2 \gamma, \qquad B \mapsto B + \gamma(1 - \mu - \gamma),
\]
for any $\gamma \in \R$. As an alternative to letting $B=0$ one may therefore consider waves of vanishing mean amplitude. 
\end{remark}

\subsection*{\sc Fourier multipliers on H\"older spaces}
The analysis given in this paper relies on certain properties of Fourier multiplier
operators given by classical symbols, and we briefly recall some of those. 
A smooth, real-valued function $a$ on $\R$ is said to be in the symbol class $S^m$ if
for some constant $c > 0$ and any non-negative integer $k$, the estimate
\[
| \partial_\xi^k a(\xi)|  \leq c (1+|\xi|)^{m-k}
\]
holds.
Analyzing the action of such operators on the Sobolev spaces $H^s$ is straightforward owing
to Parseval's formula. However, the analysis is somewhat more subtle if 
classes of continuously differentiable functions or other Banach spaces are considered. 
Although there exist natural scales of  function spaces for this class of operators (cf. \cite{MR781540}) we have chosen to work with the classical spaces $C^{s}_{\text{even}}(\s)$, $s > 0$, consisting of \(2\pi\)-periodic even functions that are $\lfloor s \rfloor$ times continuously differentiable with the $\lfloor s \rfloor$th derivative bounded 
and continuous, and H\"older continuous with H\"older exponent 
$s - \lfloor s \rfloor$ for $s \not\in \Z$. We then have the following lemma.

\begin{lemma}[\cite{MR1461211}]\label{lemma:besov}
Fix \(m \in \R\), $a \in S^m$, and let \(s \in \R_+\setminus \Z\) be a non-integral non-negative real number. Then
\[
a(\Diff) \in \LL\left( C(\s)^{s+m}, C(\s)^{s} \right)
\]
is a bounded linear operator of order \(-m\).
\end{lemma}

\begin{remark}
If one extends instead the family of H\"older spaces \(C^s\), for \(s \not\in \Z\), with the so-called Zygmund spaces for integral values of \(s\), one attains a scale of spaces for which Lemma~\ref{lemma:besov} is valid also for integral values  of \(s\)  (for more information on the relations between these different function spaces, see \cite{MR781540}). The introduction of Zygmund spaces is, however, not necessary to obtain the results in this paper.
\end{remark}


\section{Comparative local bifurcation theory}
\label{sec:local}
From now on, let \(\frac{1}{2} < \alpha  < 1\). We consider $C^{\alpha}_{\text{even}}(\s)$, the space of even and 
$\alpha$-H{\"o}lder continuous real-valued functions on the unit circle $\s$. 
The Whitham equation contains a generic nonlocal smoothing operator in the form 
of the Fourier multiplier $(\tanh(\xi)/\xi)^{1/2}$, 
whereas the corresponding KdV multiplier $1-\frac{1}{6}\xi^2$ 
is a second order differential operator. This difference may be easily overcome: rewriting the KdV equation in non-local, smoothing, form enables a similar treatment of the two equations. To illustrate how the analysis used for the Whitham equation can be applied to a larger class of equations, a comparative local bifurcation analysis is performed for the Whitham and the KdV equation (the KdV equation being a cardinal example of dispersive evolution equations and the Whitham equation's closest kin).    

\begin{theorem}[Functional-analytic formulation]
\label{thm:local}
Fix $\alpha > \frac{1}{2}$ and let $\mu > 0$.  
The solutions in $C^\alpha_{\rm{even}}(\s)$ of the Whitham equation \eqref{eq:canonical}, and those of the KdV equation \eqref{eq:kdv}, coincide with the kernel of an analytic operator 
$\Fop \colon C^\alpha_{\rm{even}}(\s) \times \R_{>0} \to C^\alpha_{\rm{even}}(\s)$ 
given by
\[
\Fop(\varphi,\mu) := \mu \varphi - \Lop({\mu})\varphi + \Nop(\varphi,{\mu}),
\]
where $\Lop(\mu)$ is bounded, linear and compact, 
and 
\[
\Nop_{\text{\tiny W}} (\varphi,{\mu}) = -\varphi^2, \qquad \Nop_{\text{\tiny KdV}}(\varphi,{\mu}) = - \Lop({\mu}) \varphi^2,
 \]
fulfills $\Diff_\varphi \Nop[0,\mu] = 0$. 
Thus $\Diff_\varphi \Fop[0,\mu]$ is Fredholm of index $0$. 
In the case of the Whitham equation the operators $\Lop$ and $\Nop$ are independent of $\mu$.
\end{theorem}

\begin{proof}
We put the KdV equation in the form
\[
\mu \varphi =  \left( 1 - {\textstyle\frac{1}{6\mu}} \partial_x^2 \right)^{-1} 
   \left( \varphi +  \varphi^2 \right), 
\]
where $\left( 1- \tau^2 \partial_x^2 \right)^{-1}  
       f := \frac{1}{2\tau} \exp\left(-\frac{|\cdot|}{\tau}\right) \ast f$ 
is the inverse of $1 - \tau^2 \partial_x^2$ on $\Schwartz^\prime(\R)$, 
the Schwartz space of tempered distributions on $\R$. 
Define
\begin{equation}\label{eq:Lops}
\Lop_\text{\tiny W} := K \ast \quad\text{ and }\quad \Lop_{\text{\tiny KdV}} 
                  :=  \left(1-{\textstyle \frac{1}{6\mu}} \partial_x^2\right)^{-1}.
\end{equation}
For functions $f \in C^\alpha_{\text{even}}(\s)$ with $\alpha > 1/2$ one has 
\[
f(x) \equiv \sum_{k \geq 0} a_k \cos(kx) \quad\text{ and }\quad  \sum_{k \geq 0} |a_k| < \infty,
\]  
meaning that $C_{\text{even}}^{\alpha}(\s)$ is a subalgebra of the Wiener algebra 
of $2\pi$-periodic functions with absolutely converging Fourier series \cite{MR2039503}. 
For such functions rudimentary calculations reduce \eqref{eq:Lops} to the intuitive formulas
\begin{align}
 \Lop_\text{\tiny W} f (x) &:= \sum_{k \geq 0} a_k 
                         \left( {\textstyle \frac{\tanh(k)}{k}} \right)^{1/2} 
                                   \cos(kx),\label{eq:lop_whitham}\\
\Lop_\text{\tiny KdV}(\mu)  f (x) &:= \sum_{k\geq 0} \frac{a_k}{1 + {\textstyle \frac{1}{6\mu}} k^2} 
                          \cos(kx),\label{eq:lop_kdv}
\end{align}
where for the Whitham equation one must use the fact that 
$K \in L^1(\R)$ (for details on $K$ and its symbol, see \cite{EK08}). 
The Fourier multiplier symbols in \eqref{eq:lop_whitham} and \eqref{eq:lop_kdv} belong to the symbol classes $S^{-1/2}(\R)$ 
and $S^{-2}(\R)$, and \eqref{eq:lop_whitham} and \eqref{eq:lop_kdv} 
are therefore bounded linear operators $C^\alpha_\text{even}(\s) \to C^{\alpha+s}_\text{even}(\s)$ 
for $s= 1/2$ and $s = 2$, respectively, see Lemma~\ref{lemma:besov}. 
Each of them is also invertible with bounded linear inverse 
$\Lop^{-1}\colon C^{\alpha+s}_\text{even}(\s) \to C^{\alpha}_\text{even}(\s)$ 
for the same values of $s$.  Due to the compactness of the embedding 
$C^\beta_\text{even}(\s) \hookrightarrow C^\alpha_\text{even}(\s)$, $\beta > \alpha$, 
both operators are compact on $C^\alpha_\text{even}(\s)$. 
We may thus define mappings 
$\Fop_\text{\tiny W}$ and $\Fop_\text{\tiny KdV} \colon C^\alpha_\text{even}(\s) \times \R_{>0} \to C^{\alpha}_\text{even}(\s)$ by
\begin{align}
\Fop_\text{\tiny W}(\varphi,\mu) 
&:= \mu \varphi -  \Lop_\text{\tiny W} \varphi -  \varphi^2 ,\label{eq:fop_whitham}\\ 
\Fop_\text{\tiny KdV}(\varphi,\mu) 
&:= \mu \varphi -  \Lop_\text{\tiny KdV}(\mu) \varphi -    \Lop_\text{\tiny KdV}(\mu) \varphi^2.
\label{eq:fop_kdv}
\end{align}
In both cases $F$ is analytic in both its arguments 
($\Lop_\text{\tiny KdV}(\mu)$ is the inverse of $1 - \frac{1}{6\mu} \partial_x^2$, which is analytic in $\mu$). 
We also have $F(0,\mu) = 0$, and the linearization $\Diff_\varphi \Fop[0,\mu] =  \mu \id - \Lop$ 
is Fredholm of index $0$ (it is a continuous perturbation of the identity by a compact operator). 
\end{proof}


\subsection*{\sc Local bifurcation}

\begin{corollary}[Local bifurcation, Whitham]
\label{cor:local_whit}
$ $\\[6pt]
{\bf (i) Subcritical bifurcation.} For each integer $k \geq 1$, there exist $\mu_k := \left({\tanh(k)}/{k}\right)^{1/2}$ 
and a local, analytic curve 
\[
\varepsilon \mapsto (\varphi(\varepsilon),\mu(\varepsilon)) \in C^\alpha_{\rm{even}}(\s) \times (0,1) 
\]
of nontrivial $2\pi/k$-periodic Whitham solutions with 
$\Diff_\varepsilon \varphi(0) = \cos(kx)$ that bifurcates from the 
trivial solution curve $\mu \mapsto (0,\mu)$ at $(\varphi(0),\mu(0)) = (0,\mu_k)$. 
In a neighborhood of the bifurcation point $(0,\mu_k)$ 
these are all nontrivial solutions of $\Fop_{\tiny{\rm W}}(\varphi,\mu) = 0$ 
in $C^\alpha_{\rm{even}}(\s) \times (0,1)$, 
and there are no other bifurcation points $\mu > 0$, $\mu \neq 1$, for solutions in $C^\alpha_{\rm{even}} (\s)$.\\[6pt] 
{\bf (ii) Transcritical bifurcation.}
At $\mu =1$ the trivial solution curve $\mu \mapsto (0,\mu)$ intersects the curve $\mu \mapsto (\mu -1, \mu)$ of constant solutions $\varphi = \mu - 1$; together these constitute all solutions in $C^\alpha_{\rm{even}} (\s)$ in a neighborhood of $(\varphi, \mu ) = (0,1)$.

\end{corollary}

\begin{proof}
The fact that $\Diff_\varphi \Fop_\text{\tiny W}[0,\mu] = \mu \id - \Lop_\text{\tiny W}$ 
is Fredholm of index $0$ and the formula~\eqref{eq:lop_whitham} show that $\mu_k$ 
are all simple eigenvalues of $\Lop_\text{\tiny W}$, 
and that no other eigenvalues $\mu > 0$ exist. 
The assertion then follows from the analytic version of the 
Crandall--Rabinowitz theorem for bifurcation from a simple eigenvalue 
\cite[Thm 8.4.1]{MR1956130}.
The fact that the solutions are $2\pi/k$-periodic can be seen by restricting attention to the subspaces 
$\{ \varphi \in C^\alpha_\text{even}(\s) \colon \varphi \text{ is } 2\pi/k\text{-periodic}\}$,
and, corresponding to the case $k = 0$, it is instantly verified that $\varphi = \mu- 1$ is a solution. By uniqueness, this family  must therefore constitute the local bifurcation curve at $\mu = 1$. Since for all other $\mu > 0$ the linearization $\Diff_\varphi \Fop_\text{\tiny W}[0,\mu]$ 
is Fredholm of index zero with trivial kernel, it is a consequence of the 
implicit function theorem that the vanishing solution is, locally, 
the unique solution in $C^\alpha_{\text{even}}(\s)$.  
\end{proof}

\begin{remark}
A more detailed proof of local bifurcation in a somewhat larger space, for $k = 1$, 
but with arbitrary wavelength, was given in \cite{EK08}. 
\end{remark}

For the KdV equation the change of variables
\begin{equation}\label{eq:KdV_change}
\varphi \mapsto k^2 \varphi(k\cdot), \qquad \mu \mapsto k^2(\mu-1) + 1,
\end{equation}  
means that the local bifurcation analysis can be reduced to the case $k=1$.

\begin{corollary}[Local bifurcation, KdV]
\label{cor:local_kdv}
Let $\mu^* := 5/6$. A local, analytic curve 
$\varepsilon \mapsto (\varphi(\varepsilon),\mu(\varepsilon)) 
               \in C^\alpha_{\rm even}(\s) \times \R_{>0}$ 
of nontrivial KdV solutions with $\Diff_\varepsilon \varphi(0) = \cos(x)$ 
bifurcates from the trivial solution curve at $(0,\mu^*)$. 
\end{corollary}
\begin{proof}
We have that
\[
\Diff_\varphi \Fop_\text{\tiny KdV} [0,\mu] \varphi 
= \sum_{k\geq 0} a_k  \left( \mu - \frac{6\mu}{6\mu+k^2} \right) \cos(kx).
\]
Therefore, $\ker  \left( \Diff_\varphi \Fop_\text{\tiny KdV}[0,\mu^*] \right) =  \spn (\cos(x))$ 
and 
\[
\Diff^2_{\varphi \mu} \Fop_\text{\tiny KdV} [0,\mu^*]\cos(x) 
=  {\textstyle \frac{5}{6}} \cos(x) \not 
          \in \range \left( \Diff_\varphi \Fop_\text{\tiny KdV}[0,\mu^*] \right).
\]
As in the proof of Corollary~\ref{cor:local_whit}, 
since $\Diff_\varphi \Fop_\text{\tiny KdV}[0,\mu^*]$ is Fredholm of index $0$,
the assumptions of the Crandall--Rabinowitz theorem are fulfilled.
\end{proof}

\begin{remark}
The bifurcation points corresponding to $k=1$ are not the same for the Whitham 
and the KdV equation: we have $\mu_1 \approx 0.87$ but $\mu^* \approx 0.83$.
\end{remark}

\begin{remark}
In the original variables, before the transformation \eqref{eq:KdV_change}, the wave speed \(\mu\) of the KdV equation becomes 
negative for $k \geq 3$, a fact related to that KdV is a proper model 
for small water waves when the wave speed is positive and the wave number small.     
\end{remark}


\section{Global bifurcation for the Whitham equation} 
\label{sec:global}
As earlier, let \(\alpha > \frac{1}{2}\) denote a fixed H\"older exponent. Let $\Fop_\text{\tiny W}$ be the Whitham operator from Theorem~\ref{thm:local}, 
defined by~\eqref{eq:lop_whitham} and~\eqref{eq:fop_whitham}. 
With 
\[
U := \left\{ (\varphi,\mu) 
      \in C^\alpha_\text{even}(\s) \times (0,1) \colon \varphi  < \mu/2 \right\},
\]
we let
\begin{equation}\label{eq:S}
S := \left\{ (\varphi,\mu) \in  U \colon  \Fop_\text{\tiny W}(\varphi,\mu) = 0  \right\}
\end{equation}
be our set of solutions. By $\BL(X)$ we denote the Banach algebra of 
bounded linear operators on a Banach space $X$, 
and by $\Lis(X)$ those which also have an inverse in $\BL(X)$, i.e., those which are Banach space isomorphisms.  

\begin{lemma}[$L^\infty$-bound]\label{lemma:bound}
Let $\mu > 0$. Any bounded Whitham solution satisfies
\begin{equation}\label{eq:bound}
\|\varphi\|_\infty \leq \mu + \| \Lop_\text{\tiny W}\|_{\LL(L^\infty(\s))}. 
\end{equation}
\end{lemma}
\begin{proof}
From $\mu \varphi - \Lop_\text{\tiny W} \varphi - \varphi^2 = 0$ we obtain that
\[
| \varphi|^2 \leq \mu |\varphi| + \| \Lop_\text{\tiny W}\|_{\LL(L^\infty(\s))} \|\varphi\|_\infty, 
\]
where we have used the fact that the Whitham kernel is integrable. Either $\varphi \equiv 0$, or we may take the supremum and divide by $\|\varphi\|_\infty$. In either case, \eqref{eq:bound} holds.
\end{proof}


\begin{lemma}[Fredholm]\label{lemma:fredholm_whitham}
The Frech\'et derivative $\Diff_\varphi \Fop_\text{\rm\tiny W}[\varphi,\mu]$ is a Fredholm operator of index $0$ for all ${(\varphi,\mu) \in U}$. 
\end{lemma}
\begin{proof}
We have
\[
\Diff_\varphi \Fop_\text{\tiny W}[\varphi,\mu] =  \left( \mu - 2\varphi \right) \id -  \Lop_\text{\tiny W}, 
\]
and, for any given $(\varphi,\mu) \in U$, that $(\mu-2\varphi)\id \in \Lis (C^\alpha_\text{even}(\s))$. In view of that $\Lop_\text{\tiny W}$ is compact on $C^\alpha(\s)$, the operator $\Diff_\varphi \Fop_\text{\tiny W} [\varphi,\mu]$ is Fredholm. The linearization $\Diff_\varphi \Fop_\text{\tiny W}[0,\mu]$ has Fredholm index zero along the trivial solution curve; we have
\[
\tau \mapsto \left( \mu - 2 \tau \varphi \right) \id - \Lop_\text{\tiny W} \in C\left([0,1],\LL(C^\alpha(\s))\right),
\]
and since the index is continuous in the operator-norm topology, it follows that it is zero also at $(\varphi,\mu)$.
\end{proof}


\begin{lemma}\label{lemma:smoothing_whitham}
Whenever $(\varphi,\mu) \in S$ the function $\varphi$ is smooth, and bounded and closed sets of $S$ are compact in $C^\alpha_\text{\rm even}(\s) \times (0,1)$.   
\end{lemma}

\begin{proof}
We write the Whitham equation in the form
\begin{equation}\label{eq:regularity}
\varphi = \tilde \Fop(\varphi,\mu) :=  \frac{\mu}{2} - \left( \frac{\mu^2}{4}  -  \Lop_\text{\tiny W} \varphi \right)^{1/2}.
\end{equation}
The mapping $\Lop_\text{\tiny W}$ is bounded and linear $C^{\alpha}(\s) \to C^{\alpha+1/2}(\s)$ (cf. Lemma~\ref{lemma:besov}), and $x \mapsto \sqrt{x}$ is real analytic for $x > 0$. Consequently, if we let
\[
V := \left\{ (\varphi,\mu) \in C^{\alpha}(\s) \times (0,1) \colon  {\textstyle \frac{\mu^2}{4}}  > \Lop_\text{\tiny W} \varphi  \right\},
\]
then $\tilde \Fop$ is real analytic $V \to C^{\alpha+1/2}(\s)$. The space $C^{\alpha + 1/2}(\s)$ is relatively compact in $C^{\alpha}(\s)$, whence $\tilde \Fop$ maps bounded subsets of $V$ into pre-compact sets.  We may then prove:\\[-6pt]

\item[] {\sc Smoothness}.
For any $\varphi \in S$  there exists a constant $R_1$ such that $\sup \varphi \leq R_1 < \mu/2$. Since $\varphi$ is a fixed point of $\tilde \Fop(\cdot,\mu)$ we have $(\varphi,\mu) \in V$.  A straightforward induction argument reveals that $\varphi \in C^\infty(\s)$.\\[-6pt]

\item[] {\sc Compactness}. 
Let $K \subset S$ be bounded and closed in the $C^{\alpha}(\s) \times \R$-topology. Then $K \subset V$, and $\{ \varphi \colon (\varphi,\mu) \in K\} = \tilde \Fop K$ is pre-compact in $C^\alpha(\s)$. Any sequence $\{(\varphi_j,\mu_j)\}_{j \geq 1} \subset K$ thus converges to a pair $(\varphi_0,\mu_0)$ in the $C^{\alpha}(\s) \times \R$-topology. The fact that $K$ is closed implies that $(\varphi_0,\mu_0) \in K$, whence $K$ is compact.

\end{proof}

\subsection*{\sc Global bifurcation} 
Using Lemmata~\ref{lemma:fredholm_whitham} and~\ref{lemma:smoothing_whitham} we shall now show that the local branches  of Whitham solutions in Corollary~\ref{cor:local_whit} can be globally extended. The following result is immediate (cf. \cite{MR1956130}) if we are able to show that, in some small neighborhood $|\varepsilon| < \delta$, $\mu(\varepsilon)$ is not identically equal to a constant.
 

\begin{theorem}[Global bifurcation]
\label{thm:global_whit}
The local bifurcation curves $\varepsilon \mapsto (\varphi(\varepsilon),\mu(\varepsilon))$ of solutions to the Whitham equation from Corollary~\ref{cor:local_whit} extend to global continuous curves of solutions $\R_{\geq 0} \to S$, with $S$ as in \eqref{eq:S}. One of the following alternatives holds:
\begin{itemize}
\item[(i)] $\| \varphi(\varepsilon)\|_{C^\alpha(\s)}$ is unbounded as $\varepsilon \to \infty$.
\item[(ii)] The pair $(\varphi(\varepsilon),\mu(\varepsilon))$ approaches the boundary of $S$ as $\varepsilon \to \infty$.
\item[(iii)] The function $\varepsilon \mapsto (\varphi(\varepsilon),\mu(\varepsilon))$ is $T$-periodic, for some $T \in (0,\infty)$.
\end{itemize}
\end{theorem}

\begin{proof}
According to \cite[Thm 9.1.1]{MR1956130}, and in view of Lemmata~\ref{lemma:fredholm_whitham} and~\ref{lemma:smoothing_whitham}, the assertion follows if any of the derivatives $\mu^{(k)}(0) \neq 0$. It will be shown in Theorem~\ref{thm:expansion} that $\ddot \mu(0) \neq 0$.
\end{proof}

In order to establish the bifurcation formulas for Whitham equation---and thereby Theorem~\ref{thm:global_whit}---we apply the Lyapunov--Schmidt reduction (cf. \cite{MR2004250}). Let $\mu^* := \mu_1$ be the bifurcation point from Corollary~\ref{cor:local_whit} and let 
\[
\varphi^*(x) := \cos(x). 
\]
Let furthermore
\[
M := \left\{ \sum_{k \neq 1} a_k \cos(kx) \in C^\alpha(\s) \right\},
\]
and
\[
N := \ker \left( \Diff_\varphi \Fop_\text{\tiny W}[0,\mu^*] \right) = \spn(\varphi^*).
\]
Then $C^{\alpha}_\text{\rm even}(\s) = M \oplus N$ and we can use the canonical embedding $C^{\alpha}(\s) \hookrightarrow L^2(\s)$ to define a continuous projection
\[
\Piop \varphi := \left\langle \varphi, \varphi^* \right\rangle_{L^2(\s)} \, \varphi^*, 
\]
with $\langle u,v \rangle_{L^2(\s)} := \frac{1}{\pi} \int_\s u v \,dx$. 

\begin{theorem}[Lyapunov--Schmidt, \cite{MR2004250}]
\label{thm:reduction}
There exists a neighborhood $\bigOh \times Y \subset U$ around $(0,\mu^*)$ in which the problem
\begin{equation}\label{eq:infinitedimenionsionalproblem}
\Fop_\text{\tiny W}(\varphi, \mu) = 0
\end{equation}
is equivalent to that
\begin{equation}\label{eq:finitedimenionsionalproblem}
\Phi(\varepsilon \varphi^*, \mu) := \Piop\, \Fop_\text{\tiny W}\left(\varepsilon \varphi^* + \psi(\varepsilon \varphi^*,\mu),\mu\right) = 0
\end{equation}
for functions $\psi \in C^\infty( \bigOh_N \times Y, M)$, $\Phi \in C^\infty( \bigOh_N \times Y, N)$, and $\bigOh_N \subset N$ an open neighborhood of the zero function in $N$. One has $\Phi(0,\mu^*) = 0$, $\psi(0,\mu^*)=0$, $\Diff_\varphi \psi(0,\mu^*) = 0$, and solving the finite-dimensional problem~\eqref{eq:finitedimenionsionalproblem} provides a solution $\varphi = \varepsilon\varphi^* + \psi(\varepsilon\varphi^*,\mu)$ of the infinite-dimensional problem~\eqref{eq:infinitedimenionsionalproblem}.
\end{theorem}

\begin{theorem}[Bifurcation formulas]\label{thm:expansion}
Let $\mu^* = \sqrt{\tanh{(1)}}$, $C_1 := (\mu^* -1)^{-1}$, and $C_2 :=  (2 \mu^* - \sqrt{ 2\tanh{(2)}})^{-1}$.
The main bifurcation curve ($k=1$) for the Whitham equation found in Corollary~\ref{cor:local_whit} satisfies 
\begin{equation}\label{eq:phiasym}
\varphi(\varepsilon) = \varepsilon \cos(x) + \varepsilon^2 \left( {\textstyle \frac{1}{2}} C_1 + C_2 \cos(2x) \right) + \bigOh(\varepsilon^3),
\end{equation}
and
\begin{equation}\label{eq:muasym}
\mu(\varepsilon) = \mu^* + \varepsilon^2  \left( C_1 + C_2 \right) + \bigOh(\varepsilon^3),
\end{equation}
in $C^\alpha_{\text{\rm even}}(\s) \times (0,1)$ as $\varepsilon \to 0$. In particular, $\ddot\mu(0) < 0$, Corollary~\ref{cor:local_whit}~(i) describes a subcritical pitchfork bifurcation, and Theorem~\ref{thm:global_whit} holds.
\end{theorem}

\begin{remark}
The solution curve $\varphi(\varepsilon)$ for the Whitham equation differs significantly from that of the KdV equation. Indeed, relying on Corollary~\ref{cor:local_kdv} it is possible to calculate the asymptotic expansion
\[
\varphi_{\text{\tiny{KdV}}}(\varepsilon) = \varepsilon \cos(x) + \varepsilon^2 \left( \cos(2x) - 3 \right) + \bigOh(\varepsilon^3),
\]
in $C^\alpha_{\text{even}}(\s)$ as $\varepsilon \to 0$ for the periodic KdV-solutions found there. Although this agrees with \eqref{eq:phiasym} to the first order in $\varepsilon$, it does not for $\varepsilon^2$. In particular, the coefficient in front of $\cos(2x)$ is positive for the KdV-solutions, whereas it is negative for the Whitham solutions.
\end{remark}

\begin{proof}
We perform first the analysis for $\mu$. We already know that $\varepsilon \mapsto \mu(\varepsilon)$ is analytic at $\varepsilon = 0$ and that $\mu(0) = \mu^*$, so it remains to show that $\dot\mu(0) = 0$ and to determine $\ddot \mu(0)$. We have
\begin{align*}
\Diff_{\varphi\varphi}^2 \Fop_\text{\tiny W}[0,\mu^*](\varphi^*,\varphi^*) &= -2 {\varphi^*}^2,\\
\Diff_{\varphi\mu}^2 \Fop_\text{\tiny W}[0,\mu^*]\varphi^* &= \varphi^*,
\end{align*}
and the value of $\dot\mu(0)$ may be explicitly calculated---for the bifurcation formulas used in this proof, we refer to~\cite[Section 1.6]{MR2004250}---as 
\[
\dot\mu(0) = -\frac{1}{2} \frac{\langle \Diff_{\varphi\varphi}^2 \Fop_\text{\tiny W}[0,\mu^*](\varphi^*,\varphi^*), \varphi^*  \rangle_{L^2(\s)}}{\langle \Diff_{\varphi\mu}^2 \Fop_\text{\tiny W}[0,\mu^*]\varphi^*, \varphi^* \rangle_{L^2(\s)}} = 0,
\]
since $\int_{\s} \cos^3(x)\,dx = 0$.

Moreover, when $\dot \mu(0) = 0$ one has that 
\begin{equation}\label{eq:DDmu}
\ddot \mu(0) = -\frac{1}{3} \frac{\left\langle \Diff_{\varphi\varphi\varphi}^3 \Phi[0,\mu^*](\varphi^*,\varphi^*,\varphi^*), \varphi^*\right\rangle_{L^2(\s)}}{\langle \Diff_{\varphi \mu}^2 \Fop_\text{\tiny W}[0,\mu^*]\varphi^*, \varphi^*  \rangle_{L^2(\s)}}.
\end{equation}
Since $\Diff_{\varphi \mu}^2 \Fop_\text{\tiny W}[0,\mu^*] = \id$ we find that the denominator is of unit size. One calculates
\begin{align*}
&\Diff_{\varphi} \Phi[\varphi,\mu]\varphi^* = \Piop \Diff_\varphi \Fop_\text{\tiny W} [\varphi + \psi(\varphi,\mu),\mu](\varphi^* + \Diff_\varphi \psi(\varphi,\mu)\varphi^*),\\[11pt] 
&\Diff_{\varphi\varphi}^2 \Phi[\varphi,\mu](\varphi^*,\varphi^*)\\ 
&= \Piop \Diff_{\varphi\varphi}^2 \Fop_\text{\tiny W} [\varphi + \psi(\varphi,\mu),\mu](\varphi^* + \Diff_\varphi \psi[\varphi,\mu]\varphi^*,\varphi^* + \Diff_\varphi \psi[\varphi,\mu]\varphi^*)\\
&\quad+\Piop  \Diff_\varphi \Fop_\text{\tiny W} [\varphi + \psi(\varphi,\mu),\mu]\Diff_{\varphi\varphi}^2 \psi[\varphi,\mu](\varphi^*,\varphi^*),
\intertext{ and, in view of that $\Fop_\text{\tiny W}$ is quadratic in $\varphi$,}  
&\Diff_{\varphi\varphi\varphi}^3 \Phi[\varphi,\mu](\varphi^*,\varphi^*,\varphi^*)\\ 
&= 3 \Piop \Diff_{\varphi\varphi}^2 \Fop_\text{\tiny W} [\varphi + \psi(\varphi,\mu),\mu](\varphi^* + \Diff_\varphi \psi[\varphi,\mu]\varphi^*,\Diff_{\varphi\varphi}^2 \psi[\varphi,\mu](\varphi^*,\varphi^*))\\
&\quad+\Piop  \Diff_\varphi \Fop_\text{\tiny W} [\varphi + \psi(\varphi,\mu),\mu]\Diff_{\varphi\varphi\varphi}^3 \psi[\varphi,\mu](\varphi^*,\varphi^*,\varphi^*).
\end{align*}
Using the form of $\Diff_\varphi \Fop_\text{\tiny W}$ together with that $\psi(0,\mu^*) = \Diff_\varphi \psi[0,\mu^*]\varphi^* = 0$ one finds that
\begin{equation}\label{eq:Dphi3}
\begin{aligned}
\Diff_{\varphi\varphi\varphi}^3 \Phi[0,\mu^*](\varphi^*,\varphi^*,\varphi^*)
&= \Piop\, \left( \mu^* \id - \Lop_\text{\tiny W} \right) \Diff_{\varphi\varphi\varphi}^3 \psi[0,\mu^*](\varphi^*,\varphi^*,\varphi^*)\\
&\quad-6 \Piop\, \varphi^* \Diff_{\varphi\varphi}^2 \psi[0,\mu^*](\varphi^*,\varphi^*).
\end{aligned}
\end{equation}
We have $\range(\mu^* \id - \Lop_\text{\tiny W}) = M$, so that $\Piop\, \left( \mu^* \id - \Lop_\text{\tiny W} \right) =0$. We thus need to determine $\varphi^* \Diff_{\varphi\varphi}^2 \psi[0,\mu^*](\varphi^*,\varphi^*)$. Since $\Diff_\varphi \Fop_\text{\tiny{W}}[0,\mu^*] = \mu^* \id - \Lop_\text{\tiny W}$ is an isomorphism on $M$, it is possible (see again \cite[Section 1.6]{MR2004250}) to rewrite $\Diff_{\varphi\varphi}^2 \psi[0,\mu^*](\varphi^*,\varphi^*)$ as
\begin{equation}\label{eq:DDphipsi}
\begin{aligned}
&\Diff_{\varphi\varphi}^2 \psi[0,\mu^*](\varphi^*,\varphi^*)\\
&=- \left( \Diff_\varphi \Fop_\text{\tiny{W}}[0,\mu^*] \right)^{-1} (\id - \Piop ) \Diff_{\varphi\varphi} \Fop_\text{\tiny W}[0,\mu^*](\varphi^*,\varphi^*)\\
&=- \left( \Diff_\varphi \Fop_\text{\tiny{W}}[0,\mu^*] \right)^{-1} (\id - \Piop ) \big(-2 {\varphi^*}^2\big)\\
&= \left( \Diff_\varphi \Fop_\text{\tiny{W}}[0,\mu^*] \right)^{-1} ( 1 + \cos(2x))\\
&= \frac{1}{\mu^* -1} + \frac{\cos(2x)}{\mu^* - \sqrt{\tanh(2)/2}}.
\end{aligned}
\end{equation}
After multiplication with $\cos(x)$ this equals
\[
\frac{\cos(x)}{\mu^* -1} + \frac{\cos(x)}{2(\mu^* - \sqrt{\tanh(2)/2})} + \frac{\cos(3x)}{2(\mu^* - \sqrt{\tanh(2)/2})}.
\]
In view of~\eqref{eq:DDmu} and~\eqref{eq:Dphi3} the coefficient in front of $\cos{x}$ equals $\frac{1}{2} \ddot\mu(0)$. All taken into consideration, we obtain \eqref{eq:muasym} via a Maclaurin series, and one easily checks that $\ddot \mu(0) < 0$.

To prove~\eqref{eq:phiasym} one makes use of the formula 
\begin{equation}\label{eq:reductionformula}
\varphi(\varepsilon) = \varepsilon \varphi^* + \psi(\varepsilon\varphi^*,\mu(\varepsilon))
\end{equation} 
from the Lyapunov--Schmidt reduction (cf. Theorem~\ref{thm:reduction}). We already know that $\varphi(0) = 0$ and $\dot\varphi(0) = \cos(x)$, so it remains to calculate $\ddot \varphi(0)$. It follows from \eqref{eq:reductionformula} that
\begin{align*}
\ddot \varphi(\varepsilon) &= \Diff_{\varphi\varphi}^2 \psi[0,\mu^*](\varphi^*,\varphi^*) + 2 \Diff_{\varphi\mu}^2 \psi[0,\mu^*] (\varphi^*, \dot \mu(0))\\ 
&\quad+  \Diff_{\mu\mu}^2 \psi[0,\mu^*](\dot \mu(0),\dot \mu(0)) + \Diff_\mu \psi[0,\mu^*]\dot \mu(0).
\end{align*}
Since $\psi(0,\mu) \equiv 0$ where $\psi$ exists, we have $\Diff_\mu \psi(0,\mu^*) = 0$. Combining this with $\dot \mu(0) = 0$ one finds that
\[
\ddot \varphi(0) = \Diff_{\varphi\varphi}^2 \psi[0,\mu^*](\cos(x),\cos(x)),
\]
so that the proposition now follows from~\eqref{eq:DDphipsi}.
\end{proof}

\subsection*{\sc Properties along the bifurcation branch} 

\begin{theorem}[Uniform convergence]\label{thm:convergence}
Any sequence of Whitham solutions $(\varphi_n,\mu_n) \in S$ has a subsequence which converges uniformly to a solution $\varphi \in C^\alpha(\s)$, with $\alpha \in (0,\frac{1}{2})$ arbitrary. At any point where $\varphi < \frac{\mu}{2}$ the solution is $\alpha$-H\"older continuous for all $\alpha \in (0,1)$.
\end{theorem}
\begin{proof}
Since $\mu_n \in (0,1)$, Lemma~\ref{lemma:bound} implies that $\varphi_n$ is uniformly bounded in $C(\s)$. The proof of Theorem 4.1 in~\cite{EK08} shows that any uniformly bounded sequence of Whitham solutions (i.e. in $L_\infty(\R)$) is equicontinuous; this follows from a standard argument using the integral kernel of $\Lop_\text{\tiny W}$. Since we are dealing with periodic solutions, we may apply the Arzola--Ascoli lemma to conclude that a subsequence of $\varphi_n$ converges in $C(\s)$.   

Since $\Lop_\text{\tiny W}$ maps $C(\s)$ into $C^{1/2}(\s)$ (cf. Lemma~\ref{lemma:besov}) we 
see from~\eqref{eq:regularity} that $\varphi \in C^{1/2}(\s)$ wherever $2\varphi \neq \mu$. On the other hand, when $\varphi(x) = \frac{\mu}{2}$ we have that
\begin{align*}
|\varphi(x) - \varphi(y)| &= \frac{\mu}{2} - \varphi(y) = \left( \frac{\mu^2}{4} - \Lop_\text{\tiny{W}} \varphi(y) \right)^{1/2}\\ 
&= \left( \Lop_\text{\tiny{W}} \varphi(x) - \Lop_\text{\tiny{W}} \varphi(y) \right)^{1/2} \leq C |x-y|^{1/4},
\end{align*}
in view of that $\Lop_\text{\tiny W} \varphi$ is $\alpha$-H\"older continuous for $\alpha = \frac{1}{2}$. Hence $\varphi \in C^{1/4}(\s)$. Repeating this argument once shows that $\varphi \in C^{3/4}$ wherever $2\varphi \neq \mu$, and $\varphi \in C^{3/8}(\s)$. Bootstrapping yields H\"older continuity below $\alpha = 1$ for $2\varphi \neq \mu$, and below $\alpha = \frac{1}{2}$ for $\varphi \in C^\alpha(\s)$.  
\end{proof}

\begin{proposition}[Characterization of blow-up]
Alternative (i) in Theorem~\ref{thm:global_whit} can happen only if 
\begin{equation}\label{eq:blow-up}
\liminf_{\varepsilon \to \infty} \, \inf_{x\in \R} \left( \frac{\mu(\varepsilon)}{2} - \varphi(x;\varepsilon) \right) = 0. 
\end{equation}
In particular, alternative~(i) implies alternative~(ii).
\end{proposition}
\begin{proof}
Assume that $\liminf_{\varepsilon \to \infty} \inf_x \left( \frac{\mu(\varepsilon)}{2} -  \varphi(x;\varepsilon) \right) \geq  \delta$, for some $\delta > 0$. Any such solution of the Whitham equation satisfies
\[
|\varphi(x) - \varphi(y)| = \frac{|\Lop_\text{\tiny W} \varphi(x) - \Lop_\text{\tiny W} \varphi (y)|}{\mu - \varphi(x) - \varphi(y)} \leq \frac{|\Lop_\text{\tiny W} \varphi(x) - \Lop_\text{\tiny W} \varphi (y)|}{2\delta}. 
\]
Since $\Lop_\text{\tiny W}$ is continuous $C(\s) \to C^{1/2}(\s)$ and the family $\{\varphi(\varepsilon)\}_\varepsilon$ is uniformly bounded in $C(\s)$ (cf. Lemma~\ref{lemma:bound}), it follows that $\{\varphi(\varepsilon)\}_\varepsilon$ is uniformly bounded in $C^{1/2}(\s)$ too. Repeating the argument for $\Lop_\text{\tiny W}$ as a continuous operator $C^{1/2}(\s) \to C^\alpha(\s)$, $\alpha < 1$, yields that 
\[
\| \varphi(\varepsilon) \|_{C^{\alpha}(\s)} \leq C \delta^{-2}, \qquad \alpha \in (0,1),
\]
for some constant $C$ depending only $\Lop_\text{\tiny W}$ (recall that $\mu$ is bounded by assumption). This shows that $\| \varphi(\varepsilon)\|_{C^{\alpha}(\s)} \to \infty$ is possible only if \eqref{eq:blow-up} holds.
\end{proof}

\subsection*{Summary} In Section~\ref{sec:global} we have shown that each of the subcritical bifurcation branches in Corollary~\ref{cor:local_whit} (i) can be extended to a global bifurcation branch in the set \(S\) of solution pairs \((\varphi,\mu) 
      \in C^\alpha_\text{even}(\s) \times (0,1)\) satisfying \(\varphi  < \mu/2\). These solutions are all smooth. Each bifurcation curve contains a sequence of solutions which converge in \(C^0\) to a solution of H\"older regularity \(\alpha \in (0,\frac{1}{2})\), with the better pointwise regularity \(\alpha \in (0,1)\) at any point \(x\) where \(\varphi(x) < \mu/2\). If the \(C^\alpha\)-norm, \(\alpha > \frac{1}{2}\), blows up along the bifurcation branch, the corresponding sequence of solutions approaches \(\varphi = \frac{\mu}{2}\) in the sense of \eqref{eq:blow-up}. In spite of this---and in spite of the numerical evidence available---an analytical proof of a 'highest', cusped, limiting wave at the end of each bifurcation branch remains an open problem.

\section{Spectral schemes for the Whitham equation}
\label{sec:spectral}
In this section, the numerical approximation of
solutions of \eqref{eq:whit} is in view. Both for the approximation of traveling waves and
for the discretization of the evolution problem, spectral schemes are used. This is a natural
choice since the operator defining the linear part of the equations is given in terms
of Fourier multipliers.

\subsection*{Traveling waves}
Solutions of  \eqref{eq:canonical} are approximated
using a spectral cosine collocation method.
To define the cosine-collocation projection,
first define the subspace
\[
S_N = \realspan \left\{ 
\cos(lx)\ \Big| \  0 \le l  \le N-1 \right\}
\]
of $L^2(0, \pi)$, and
the collocation points
$x_n = \pi\frac{2n-1}{N}$ for $n=1,...,N$.
The discretization is defined by seeking $\phi_N$ 
in $S_N$ satisfying the equation
\begin{equation}\label{steady:colloc} 
- \mu \, \phi_N + \phi_N^2 + K^N \phi_N =0,
\end{equation}
where the operator $K^N$ the discrete form of $K$ defined 
in \eqref{eq:K1}.
The discrete cosine representation of $\phi_N$ given by
\[
\phi_N(x) = \sum_{l=0}^{N-1} w(l) \Phi_N(l) \cos(lx),
\]
where 
\[
w(l) = \left\{ \begin{array}{cc} \sqrt{1/N}, \ \ & l = 0, \\
                                 \sqrt{2/N}, \ \ & l \ge 1,   
                \end{array} \right.
\]
and $\Phi_N(l)$ are the discrete cosine coefficients, given by
\[
\Phi_N(l) = w(l) \sum_{n=0}^{N-1}\phi_N(x_n) \cos(lx_n), \quad \mbox{ for } l=0,...,N-1.
\]
Now if the equation \eqref{steady:colloc}
is enforced at the collocation points $x_n$, the term
$K^N \phi_N$ may be practically evaluated 
with the help of the matrix $\left[ K^N \right](m,n)$
by
\[
\left[ K^N\right] \phi_N (x_m) 
= \sum_{n=1}^{N} \left[ K^N \right](m,n) \phi_N(x_n),
\]
where the matrix $\left[K^N \right](m,n)$
is defined by
\[
 \left[ K^N\right](m,n)
=
\sum_{l =0 }^{N-1} w^2(k)
\sqrt{   {\textstyle \frac{1}{l}} \tanh{l}} \ 
\cos(l x_n) \cos(l x_m).
\]
Thus, equation \eqref{steady:colloc} enforced at the collocation 
points $x_n$ yields a system of $N$ nonlinear equations, 
which can be efficiently solved using a Newton method. 
The cosine expansion
effectively removes the singularities of the Jacobian matrix
due to translational invariance and symmetry of the solutions.
The nondimensional speed $\mu$ is used as the bifurcation parameter,
but as the turning point is approached, it is more convenient to
change to a waveheight parametrization. 
The computation can be started for $\mu$ close to but smaller
than the critical speed $\mu_k$, 
and with an initial guess for the 
first computation given by comparing \eqref{eq:phiasym} and \eqref{eq:muasym}.
\begin{figure}
  \begin{center}
    \includegraphics[width=.49\textwidth]{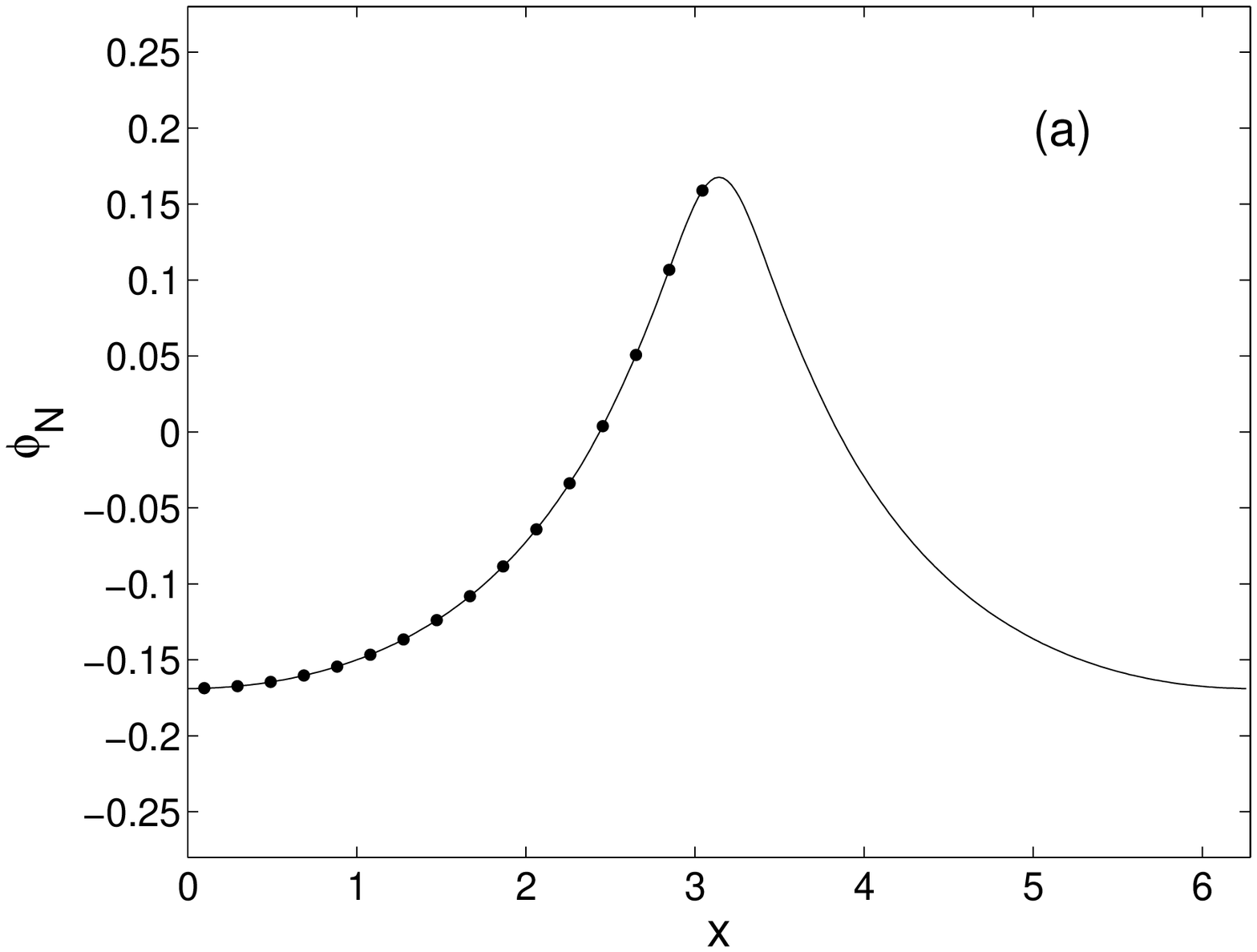}
    \includegraphics[width=.49\textwidth]{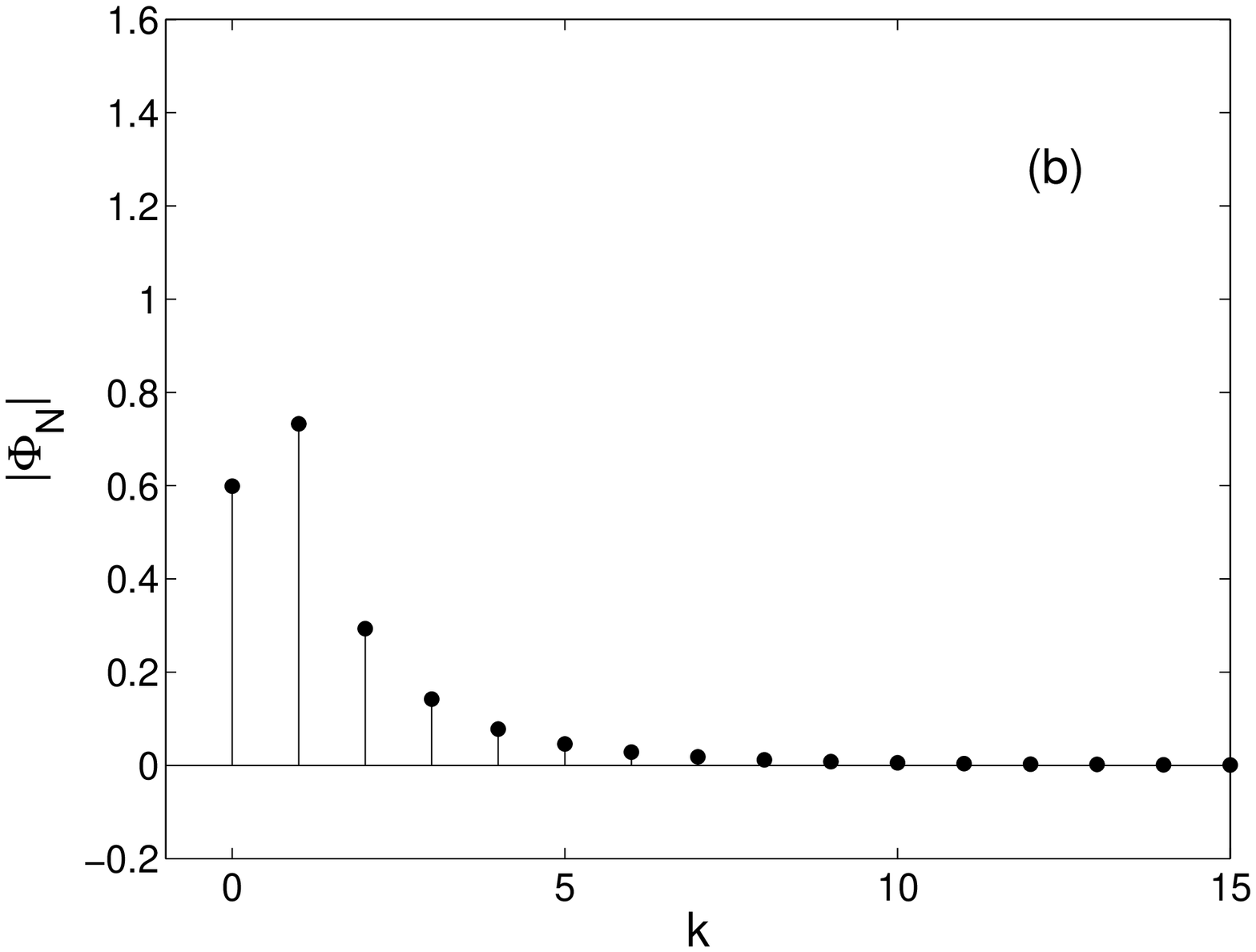}
  \end{center}
  \caption{\small An approximate traveling-wave solution $\phi_N$ with nondimensional wavespeed 
    $\mu = 0.789$, wavelength $2 \pi$, and nondimensional waveheight $0.3368$
    is shown in the left panel.
    The collocation approximation used $16$ gridpoints, which are indicated as dots in the graph.
    The amplitudes of the discrete Fourier modes are shown on the right.}
\label{fig:1}
\end{figure}
\begin{figure}
  \begin{center}
    \includegraphics[width=.49\textwidth]{Figure2aNEW.ps}
    \includegraphics[width=.49\textwidth]{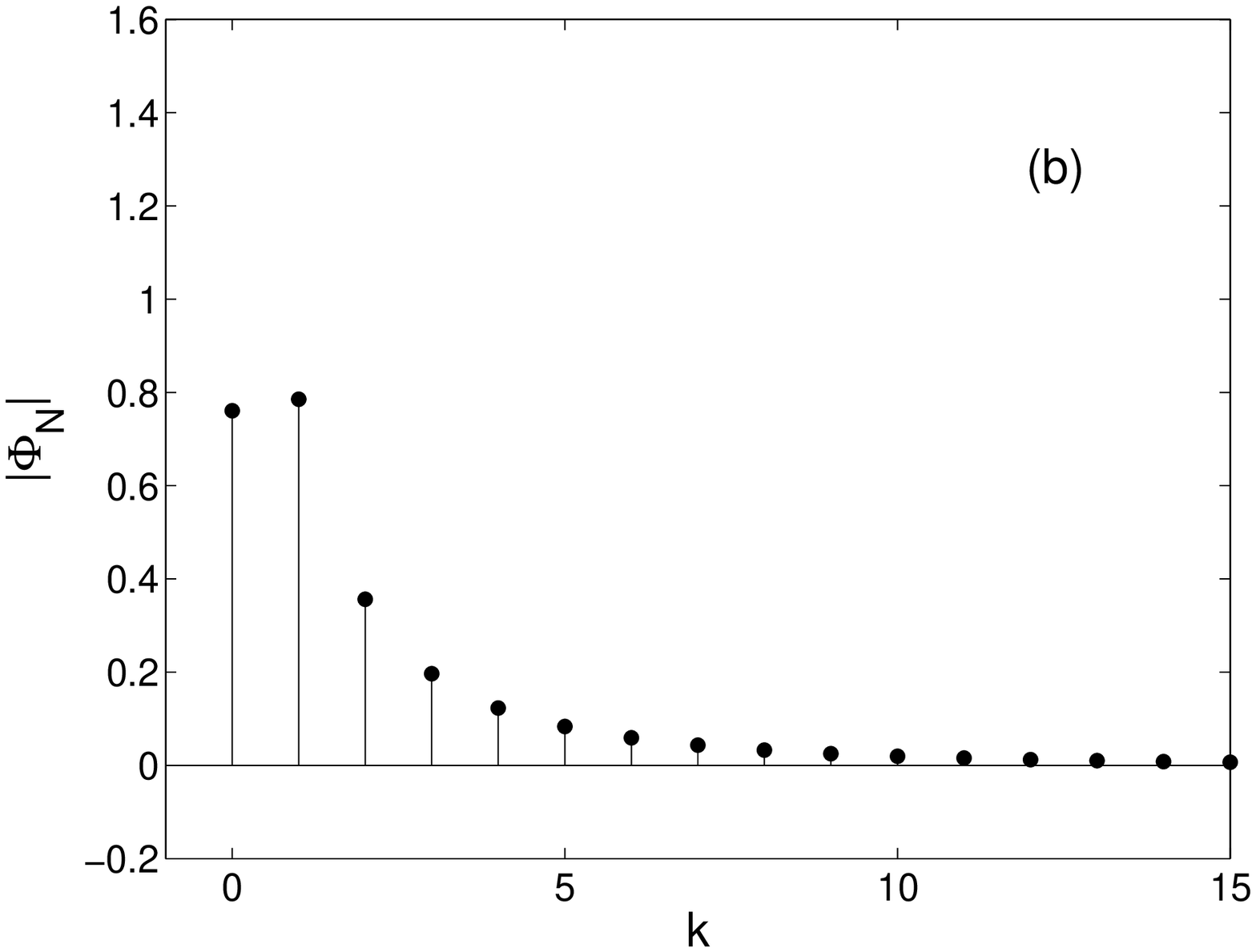}
  \end{center}
  \caption{\small An approximate  traveling-wave solution $\phi_N$ with 
    nondimensional wavespeed $\mu = 0.7715$, wavelength $2\pi$,
    and nondimensional waveheight $ 0.4152$.
    The collocation approximation used $64$ gridpoints, 
    and the amplitudes of the discrete Fourier modes are shown on the right.}
\label{fig:2}
\end{figure}

In Figure \ref{fig:1}, a periodic traveling wave with speed $\mu=0.789$ is shown.
In panel (a), the function is drawn on the interval $[0,2 \pi]$,
and the collocation values are indicated with dots in the curve.
In panel (b), the values of the discrete cosine coefficients
are shown. This computation was done with only $16$ modes.
In Figure \ref{fig:2}, a periodic traveling wave with speed $\mu=0.7715$ is shown.
It appears that the wave is steeper, and higher cosine coefficients
are important. The computation was done with $64$ modes.
To make sure that the computed functions are approximate traveling
waves for the Whitham equation, we have also used a dynamic integrator
for the time-dependent Whitham equation, as shown in the next subsection.

\subsection*{\sc Time-dependent discretization}
%
For the purpose of approximating time-dependent periodic solutions of \eqref{eq:whit}, 
a Fourier method is optimal.
To define the Fourier-collocation projection,
the subspace
\begin{equation}
\Sn = \compspan \left\{ 
\exp(ikx)\ \Big| \  k \in {\mathbb Z}, \  - N/2 \le k \le N/2 -1 \right\}
\end{equation}
of $L^2(0, 2\pi)$ is used.
The collocation points are defined to be
$x_j = \frac{2 \pi j}{N}$ for $j=0,1,...,N-1$.
Note that $N$ here is in general different from the $N$
used in the previous section.
Let $P_N$ be the projection operator onto $\Sn$, 
and let $I_N$ be the interpolation operator 
from $C^{\infty}_{per}([0,2 \pi])$ onto $\Sn$.
As explained in \cite{MR1874071,0658.76001}, 
this operator is defined for a
given $u \in C^{\infty}_{per}([0,2 \pi])$,
by letting $I_N u$ be the unique element of $\Sn$ that coincides
with $u$ at the collocation points $x_j$.
The same scaling as in the steady case is used, so that the non-dimensional form
of the equation is
\begin{equation}
\label{eq:dynamic}
\eta_t + 2 \eta \eta_x + K \ast \eta_x = 0.
\end{equation}
The spatial discretization is then defined by the following problem.
Find a function $\eta_N:[0,T] \rightarrow \Sn$,
such that
\begin{equation} \label{colloc}
\left\{
\begin{array}{ll}
\partial_t \eta_N + \partial_x I_N ( \eta_N^2)  
                 + \mathcal{K}^N \! * \partial_x \eta_N 
                 = 0, & x \in [0,2\pi],\\
 \eta_N(\cdot,0) = P_N \eta_0, & \\
\end{array}
\right.
\end{equation}
where $\eta_0$ is the initial data.
If $\eta_N(\cdot,t)$ is written in terms of its discrete Fourier coefficients
$\tilde{\eta}_N(k)$ as
\[
\eta_N(x) = \sum_{-N/2 \le k \le N/2 -1 }\tilde{\eta}_N(k)
\exp(ikx),
\]
the operator $\mathcal{K}^N$ can be evaluated using
the formula
\[
\mathcal{K}^N \eta_N (x) = 
\tilde{\eta}_N(0) \ +  
\sum_{\substack{1-N/2 \le k \le N/2 -1 \\ k \ne 0 }}  
  \sqrt{   {\textstyle \frac{1}{k}} \tanh{k}} \ \tilde{\eta}_N(k) \exp(ikx).
\]
Thus the operator  $\mathcal{K}^N$ is 
the truncation at the $N/2$-st Fourier mode of the 
operator given by the periodic convolution with $K$.
Note that this formulation includes the truncation
of the Fourier mode $\tilde{\phi}_N(-N/2)$ which ensures
that the solution stays real, and which
otherwise can lead to instabilities in the computation.
The time integration of \eqref{eq:dynamic} may be carried out in a number of ways.
We have chosen to use a midpoint method for both the
linear and the nonlinear term.
This method can be motivated as follows. Suppose the discrete solution
is known at the time $t$. Then by the fundamental theorem of calculus,
it is clear that
$$
\eta_N(x_j, t+\delta t) = \eta_N(x_j,t) + \int_t^{t+\delta t} \partial_t \eta(x_j,s) \, ds.
$$
Now using the equation \eqref{colloc} 
and approximating the integral with help of the trapezoidal rule yields
\begin{multline*}
\eta_N(x_j, t+\delta t) = \eta_N(x_j,t)
     + \frac{\delta t}{2} \left\{ 
       \partial_x I_N \big( \eta_N^2(x_j,t) + \eta_N^2(x_j, t+\delta t)\big) 
                                                                             \right\} \\
          + \frac{\delta t}{2} \left\{ \mathcal{K}^N \eta_N(x_j,t) 
                     + \mathcal{K}^N  \eta_N (x_j, t+\delta t)
\right\} + 
\mathcal{O}(\delta t^3).
\end{multline*}
From this equation, one may easily construct a fully discrete approximation.
This can be done either in physical space, or in Fourier space. Constructing
the approximation in Fourier space appears to be more convenient for the
time-dependent problem. Since nonlinear terms are involved, the method
is in effect pseudo-spectral, that is, derivatives and integral operators
are evaluated in Fourier space, while nonlinear terms are taken in physical
variables. In practice, the term $\eta_N^2(x_j, t+\delta t)$ is approximated
by $\eta_N^2(x_j, t)$, and the equation is solved for  $\eta_N(x_j, t+\delta t)$.
Then this is used as a new approximation for the square term.
This procedure is iterated until the change in the result falls below a certain
tolerance. This scheme is second-order in time, so that a relatively small time
step has to be used. Of course, this also means that the number of iterations
at each time step is small, typically less than three.  
%
%
%
\begin{table}
\caption{\small Error in evolution code after 5 and 50 periods for the traveling wave
           shown in Figure \ref{fig:1}.}   
\begin{center}
%
%
\begin{tabular}[c]{c c c c c c c c c c c} 
 \hline
 N & & $\delta t$ & & $L^2$-error & &  $|u|_{\infty} - |u_N|_{\infty}$  & & $L^2$-error & & $|u|_{\infty} - |u_N|_{\infty}$ \\
  \hline
  $2^{5}$ & & $2^{-10}$ & & 2.675e-04  & & 2.838e-05 & &  9.635e-04 & & 2.612e-05 \\
  $2^{6}$ & & $2^{-14}$ & & 2.644e-05  & & 1.876e-06 & &  5.299e-05 & & 1.587e-06 \\
  $2^{7}$ & & $2^{-19}$ & & 1.021e-06  & & 5.876e-08 & &  4.489e-07 & & 5.286e-08    \\
  \hline
   & &    & &  \multicolumn{3}{c}{$1$ period} & & \multicolumn{3}{c}{$100$ periods} \\
  \hline
\end{tabular}    
\end{center}
\label{tab:1}       
\end{table}
\begin{figure}[h]
  \begin{center}
    \includegraphics[width=.49\textwidth]{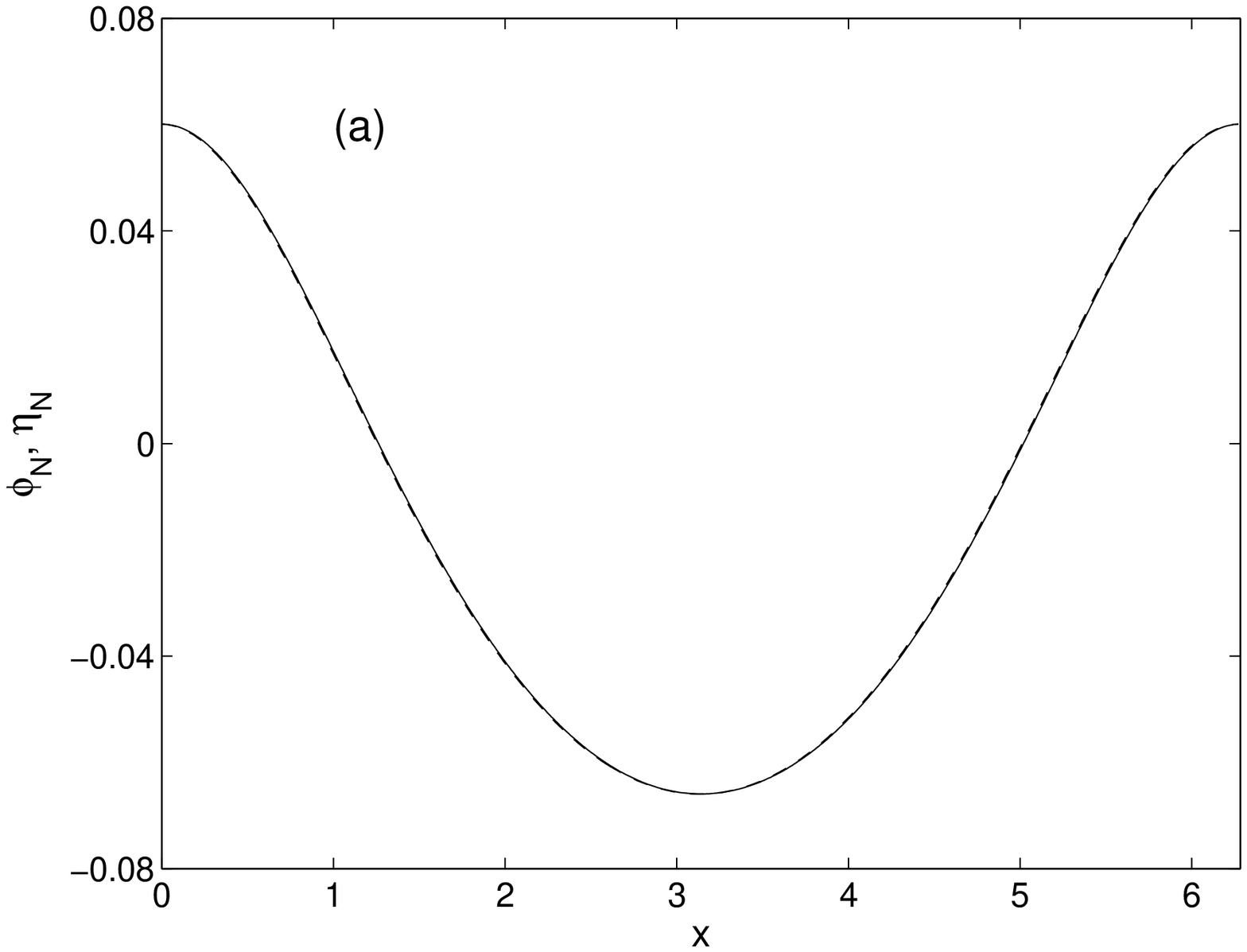}
    \includegraphics[width=.49\textwidth]{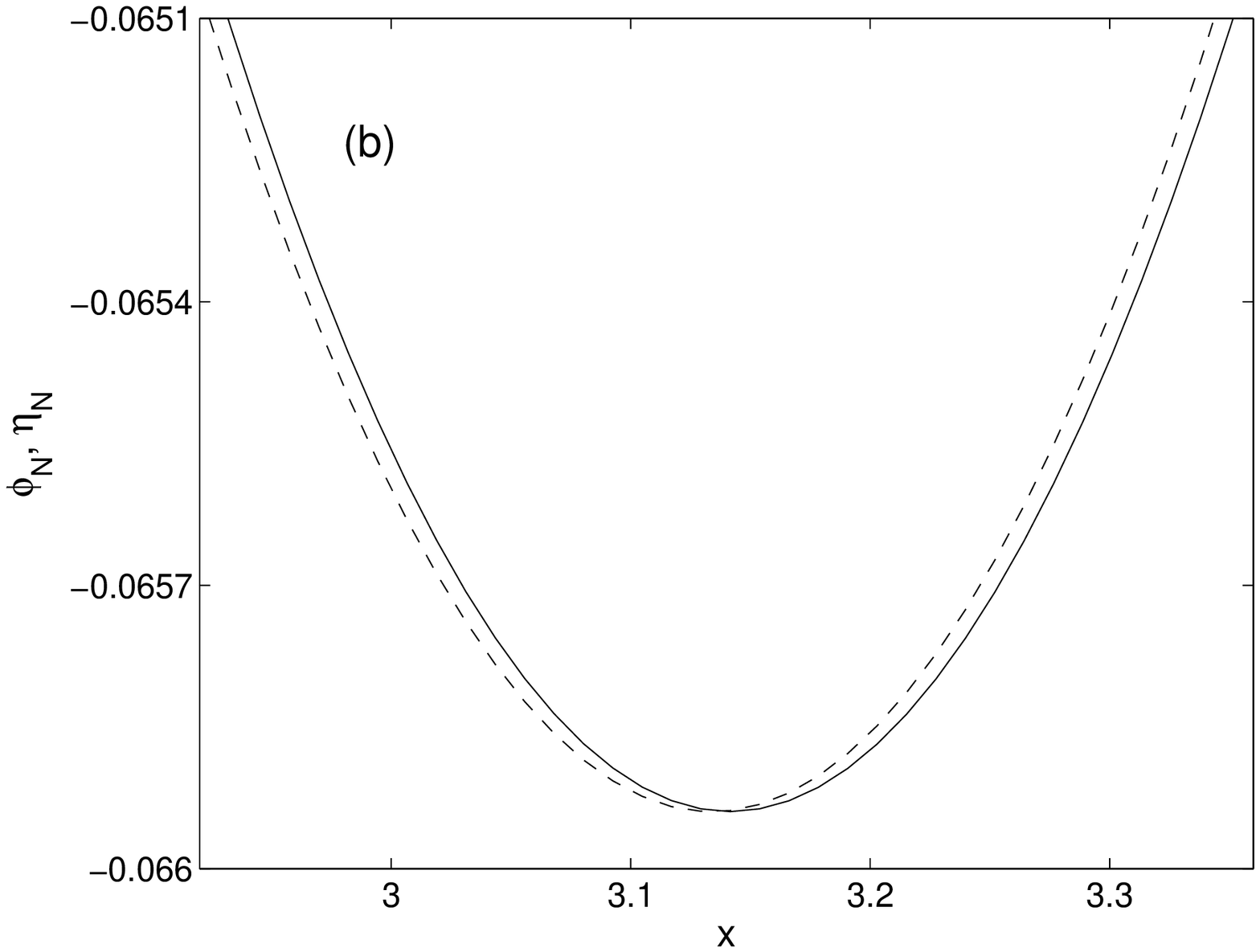}
  \end{center}
  \caption{\small Solid line: approximate traveling wave $\phi_N$ for the Whitham equation
            $\mu=0.8615$. Dashed line: $\eta_N$ after time integration using \eqref{colloc} 
            for $10000$ periods.
            In (a), the difference between $\phi_N$ and $\eta_N$ is hardly visible.
            In this computation, $N=512$ and $\delta t=5e-04$. The $L^2$ error was 
            $2.1e-03$,
            while the difference in height was $2.24e-06$. 
            This and the magnification (b) suggests that the
            the error is mainly due to a phase shift.}
\label{fig:3}
\end{figure}
In Table 1, we record the numerical errors incurred by the
time integration of an approximate traveling wave with 
velocity $\mu=0.8615$ propagating for $1$ and $100$ periods.
To find the most advantageous combinations of the number
of Fourier modes $N$ and the time step $\delta t$, we used a computation
for one period. We then use this combination, 
and integrated for $100$ periods.
The discrete $L^2$-error, the difference in maximal height
between the original wave, and the profile after $1$ and $100$
periods were computed. 
As can be seen, the error drops when increasing $N$
or when decreasing $h$. Moreover, the fact that the difference in
maximal height is generally smaller than the $L^2$-error
suggests that the error incurred during the time evolution
is mostly due to a phase shift of the solution.
This can also be observed in Figure \ref{fig:3}, where the same
traveling wave is shown after time integration for $10000$
periods. These results also suggest that the traveling
waves are orbitally stable, but no special investigation
of this question has been carried out.

\section{Numerical results}
\label{sec:results}
In this section, we present a few numerical computations carried 
out using the discretizations described above. First, we will
present an in-depth study of some bifurcation branches of the
Whitham equation, and then the Whitham waves will be compared
to traveling-wave solutions of the KdV equation.

\subsection*{\sc The Whitham branch}
\begin{figure}[h]
  \begin{center}
    \includegraphics[width=.49\textwidth]{Figure3ato4a.ps}
    \includegraphics[width=.49\textwidth]{Figure3bto4bNEW.ps}
  \end{center}
  \caption{\small Bifurcation branch of $2 \pi$-periodic traveling wave solutions of the 
   Whitham equation. In (a), the entire bifurcation branch is indicated with dots, 
   and compared to a curve described by the expressions given in \eqref{eq:muasym}
   and \eqref{eq:phiasym}. 
   In (b), a close-up of the turning point is shown.}
\label{fig:4}
\end{figure}
\begin{figure}[h]
  \begin{center}
    \includegraphics[width=.49\textwidth]{Figure4a.ps}
    \includegraphics[width=.49\textwidth]{Figure4b.ps}
  \end{center}
  \caption{\small
This figure presents computations using the dynamic discretization with
approximate Whitham traveling waves as initial data. The computations were
done with $N=256$ and time step $\delta t = 0.001$. The initial data  
$\phi_N$ is represented by a dashed line, and the computed profile
$\eta_N$ after one period is represented by a solid line.
In (a), a wave with wavespeed $c= 0.7664$ is shown,
In (b), a wave with wavespeed $c= 0.7667$ is shown.}
\label{fig:5}
\end{figure}
\begin{figure}[b]
  \begin{center}
    \includegraphics[width=.49\textwidth]{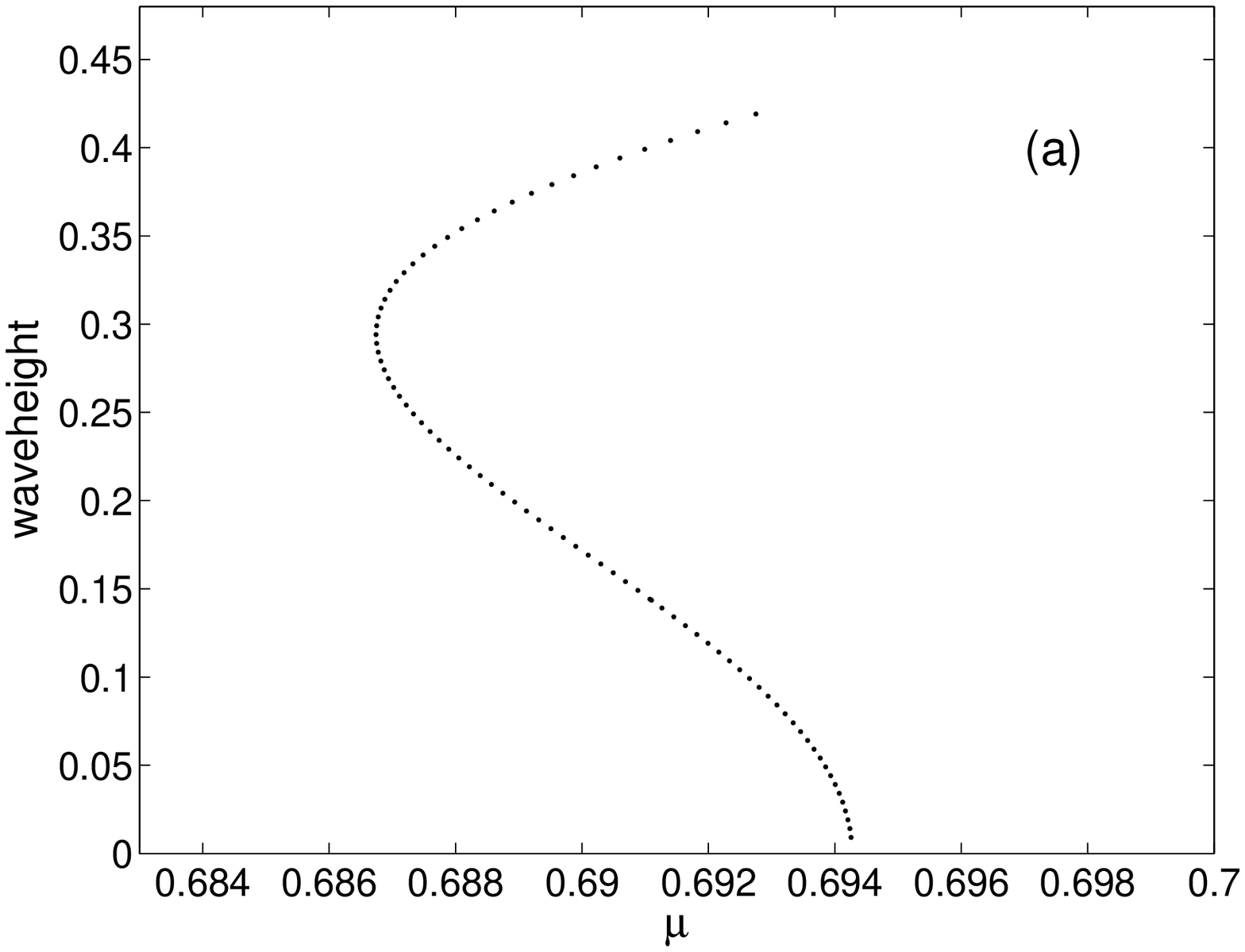}
    \includegraphics[width=.48\textwidth]{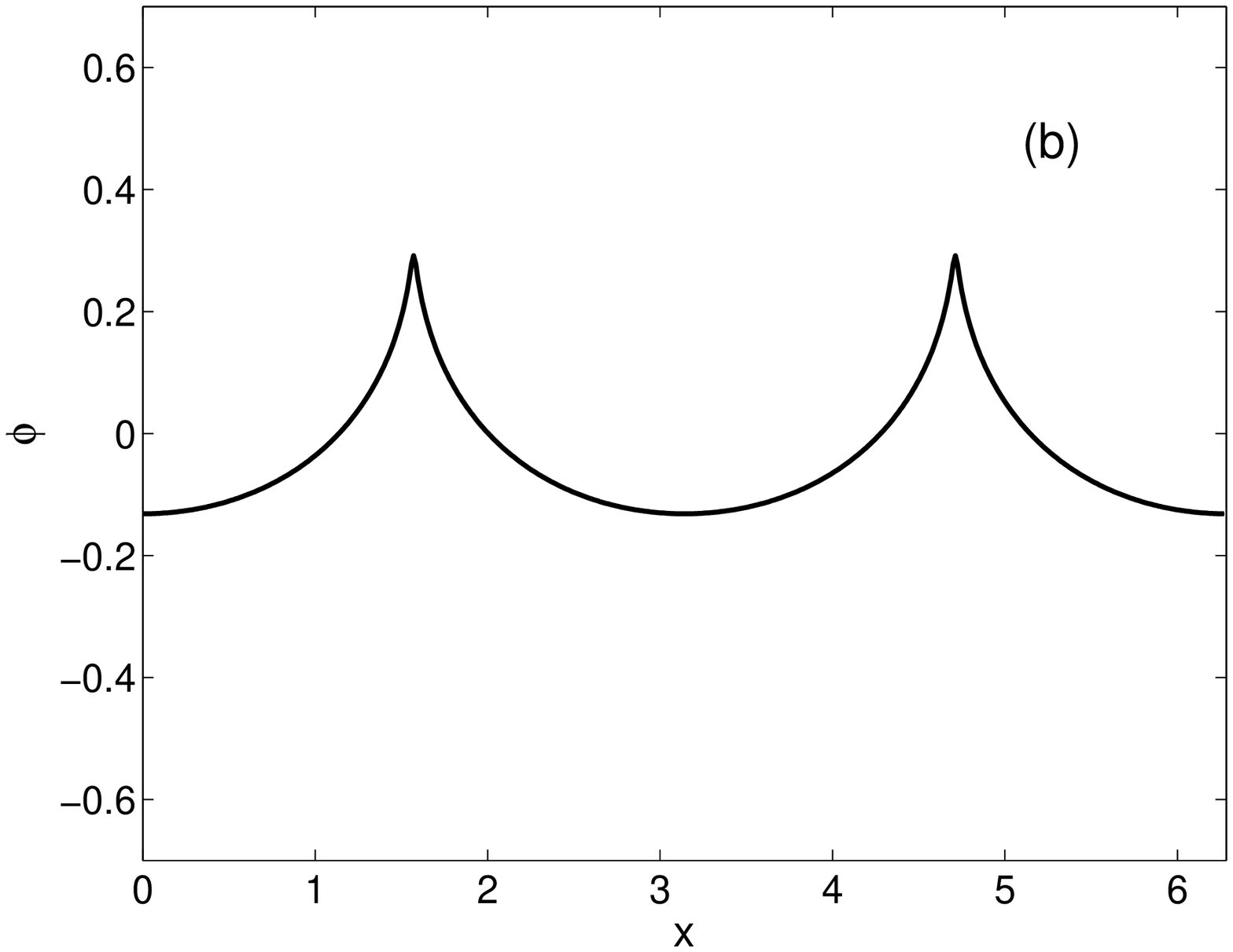}
  \end{center}
  \caption{\small A branch of approximate traveling waves for the Whitham equation with $k=2$
is shown in (a). The profile of the highest wave on $[0,2\pi]$ is shown in (b). 
}
\label{fig:6}
\end{figure}
In Figures \ref{fig:1} and \ref{fig:2}, 
traveling wave solutions of the Whitham equation were shown.
Figure \ref{fig:4} shows the bifurcation branch of these waves.
In Figure \ref{fig:4}(a), the whole branch is shown, and compared to the formula
given in \eqref{eq:muasym}. In Figure \ref{fig:4}(b) a close-up of the turning point
is shown.
Determining the endpoint of the branch is challenging. Here a combination of
two tests are used. First, if a point is suspected to be higher than the
highest branchpoint, then the solution is interpolated, and the Newton
scheme run on a finer grid. Failure of this scheme to converge indicates
that the point does not lie on the branch. 
The highest point in Figure \ref{fig:4}(b) corresponds 
to a wave for which the Newton scheme converged with more
than $1400$ gridpoints on the interval $[0,2 \pi]$.
The second test uses the 
dynamic integrator. 
For example, Figure \ref{fig:5}(a) shows the results of running
the wave corresponding to the highest point on the branch shown in Figure 3
through the dynamic integrator for one period. 
Figure \ref{fig:5}(b) shows the next point
(which is not shown in Figure \ref{fig:4}), but which can be computed with lower
values of $N$ (up to about $N=400$).
The difference is dramatic, and the time integration cannot be rescued
by choosing a smaller time step. Nevertheless, we do not expect the numerical
scheme to be able to approximate the wave corresponding to the very highest point on the branch.
Indeed, comparing the maximum of the highest wave shown in Figures \ref{fig:4}
and \ref{fig:6} 
to the statement of Theorem \ref{thm:convergence}, it appears that the branch
of H\"{o}lder continuous solutions should continue further than shown in 
those figures. On the other hand, our criterion for termination of the
numerical bifurcation included failure of the time-dependent scheme to converge,
and this failure can also be a result of instability of the computed traveling wave.
Moreover, if the highest wave does indeed feature a cusp, 
then the spectral scheme will fail to converge, and convergence for
waves close to the terminal wave will also be delicate.
A method for computing cusped waves was proposed in \cite{1027.65136},
and this might be a good approach for possilbe further studies.

In Figure \ref{fig:6}, a branch with $k=2$ is shown. To find the highest point on
this branch, similar methodology as for the principal branch was used. We note that
the turning point is much farther from the end of the branch than in the principal branch.

\subsection*{\sc Comparison with KdV}
%
\begin{figure}[b]
  \begin{center}
    \includegraphics[width=.49\textwidth]{Figure7a.ps}
    \includegraphics[width=.48\textwidth]{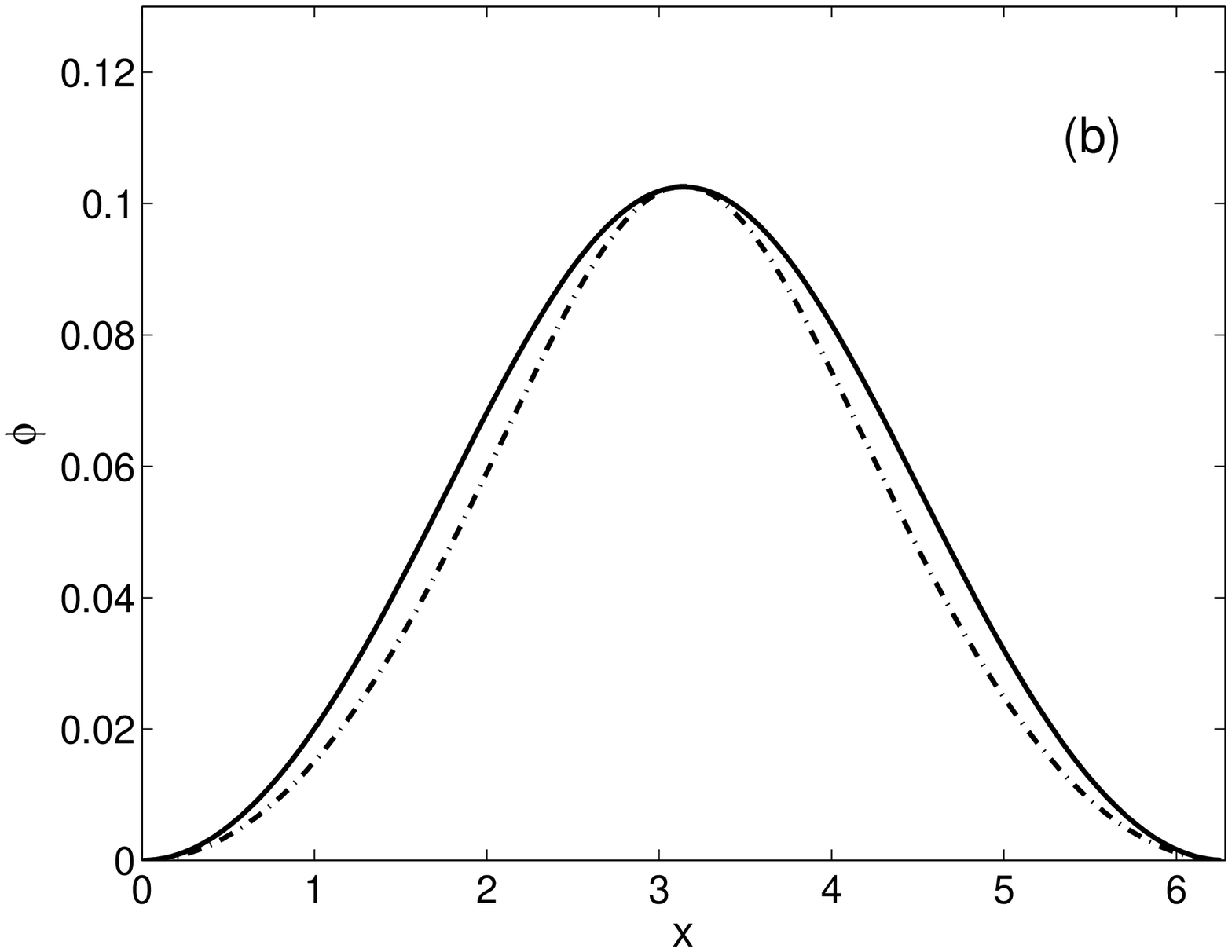}
    \includegraphics[width=.48\textwidth]{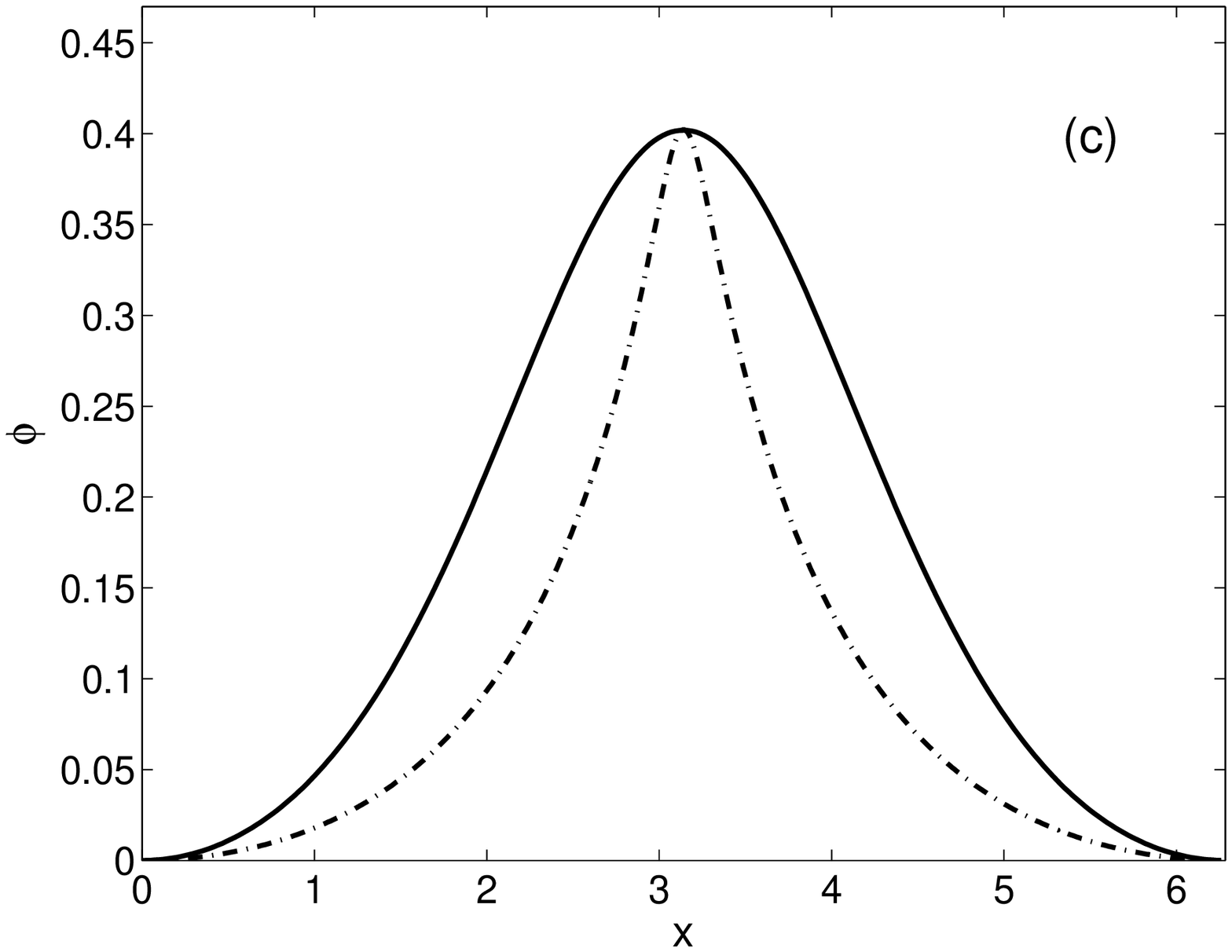}
    \includegraphics[width=.48\textwidth]{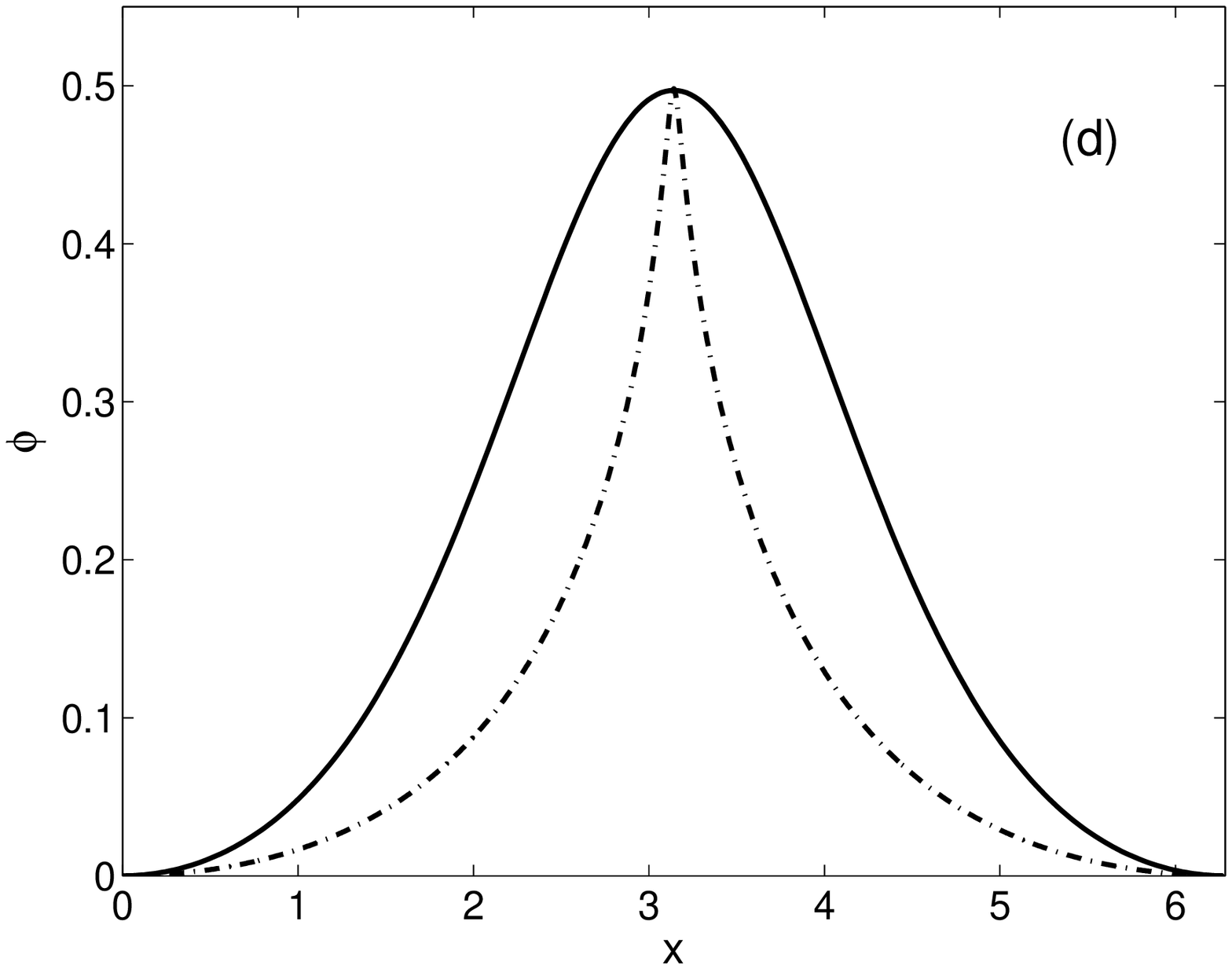}
  \end{center}
  \caption{\small A branch of approximate traveling waves for the Whitham equation with $k=1$
(the principal branch) and a branch of traveling waves for the KdV equation are shown in (a). 
Comparisons between the shapes of Whitham and KdV traveling waves are shown in (b), (c)
and (d). The waves are arranged so that the minimum coincides with the $x$-axis.
The Whitham waves are represented by dashed-dotted curves, and the KdV waves
are represented by solid curves.
The waveheight of the waves is $0.1$ in (b), $0.4$ in (c) and $0.485$ in (d). 
}
\label{fig:7}
\end{figure}
It is of interest to see how the change in the linear dispersion
relation affects the shape of the traveling waves.
Therefore, attention will now be focused on the comparison of traveling waves
of the Whitham equation to traveling-wave solutions of the KdV equation.

First the bifurcation branches are compared. For this purpose, the wavelength
is held fixed at $2 \pi$. 
It was noted in Section \ref{sec:local} that the bifurcation points
$\mu_1$ and $\mu^*$ are not the same. 
This difference is due to the small wavelength $2\pi$. A comparison
using a larger wavelength would yield a smaller difference in bifurcation points.
Figure \ref{fig:7}(a) shows bifurcation branches for both KdV (solid curve)
and Whitham (dotted curve) equations. 
As noted above, the Whitham branch appears to have a highest wave,
after which the approximation no longer converges.
On the other hand, the KdV branch is unlimited, and
after a Galilean shift, this branch will approach a solitary wave.

For the purposes of comparing the shape, we chose to fix the waveheight 
and wavelength of the traveling waves. Other normalizations
are of course possib;e, such as fixing the speed or some integral measure.
In Figure \ref{fig:7}(b), (c) and (d) Whitham and KdV traveling waves
with respective waveheights of $0.1$, $0.4$ and $0.485$ are shown. 
For a better comparison, the waves are plotted with the minimum at zero.
The small-amplitude waves are very similar in shape. However, for the
large-amplitude waves we note that the Whitham waves appear narrower, with
a more pronounced and sharp peak.

\section{Conclusion}\label{sec:conclusion}
In this article, a bifurcation theory has been developed to prove existence of a
global branch of traveling-waves solutions of the Whitham equations \eqref{eq:whit}.
Smoothness of the solutions lying on the bifurcation curve
has been proved, and it has been shown that the solutions converge uniformly to a solution of H\"{o}lder regularity $\alpha \in (0,1)$, expect possibly at the highest crest point (where $\alpha$ may be less than $1/2$). The larger H\"{o}lder norm blows up only if the highest point is attained. It was conjectured already by Whitham \cite{MR1699025} that this branch of
traveling wave has a limiting wave which has a cusp. Unfortunately such a result
has not been established, but remains an interesting open problem.
A comparative local bifurcation was provided for the Whitham and the KdV equations,
and numerical tools were used to exhibit properties of the principal branch ($k=1$)
and the secondary bifurcation branch ($k=2$) of the Whitham equation.
Finally, a comparison was made between the Whitham and KdV traveling waves which
revealed that large-amplitude Whitham waves are narrower than their KdV counterparts.
It remains to be seen which of these waves resembles the shape of a traveling water wave
(the solution of the free-surface boundary-value problem for the Euler equations) 
more closely.

\begin{acknowledgments}
The authors would like to thank Erik Wahl{\'e}n for pointing out an error in an earlier preprint. This research was supported in part by the Research Council of Norway. 
\end{acknowledgments}
%

\bibliographystyle{siam}
\bibliography{whitham}

\end{document}